\documentclass[11pt]{article}

\usepackage{amsmath}
\usepackage{amsfonts}
\usepackage{color}
\usepackage{amssymb}
\usepackage{graphicx}
\usepackage{enumerate}
\usepackage{amsthm,amscd}
\usepackage[all]{xy}

\allowdisplaybreaks

\usepackage{hyperref}

\newtheorem{theorem}{Theorem}[section]

\newtheorem{corollary}[theorem]{Corollary}
\newtheorem{definition}[theorem]{Definition}

\newtheorem{lemma}[theorem]{Lemma}
\newtheorem{proposition}[theorem]{Proposition}

\newtheorem{remark}[theorem]{Remark}
\newtheorem{clm}[theorem]{Claim}
\newtheorem{problem}[theorem]{Problem}

\begin{document}

\title{Non-empty pairwise cross-intersecting families\thanks{This work is supported by  NSFC (Grant No. 11931002 and 12371327).  E-mail addresses: yangmiemie@hnu.cn (Yang Huang),
ypeng1@hnu.edu.cn (Yuejian Peng, corresponding author).}}

\author{Yang Huang, Yuejian Peng$^{\dag}$ \\[2ex]
{\small School of Mathematics, Hunan University} \\
{\small Changsha, Hunan, 410082, P.R. China }  }

\maketitle

\vspace{-0.5cm}

\begin{abstract}
Two families $\mathcal{A}$ and $\mathcal{B}$ are cross-intersecting if $A\cap B\ne \emptyset$ for any $A\in \mathcal{A}$ and $B\in \mathcal{B}$. We call $t$ families  $\mathcal{A}_1, \mathcal{A}_2,\dots, \mathcal{A}_t$ pairwise cross-intersecting families if $\mathcal{A}_i$ and $\mathcal{A}_j$ are cross-intersecting when $1\le i<j \le t$. Additionally, if $\mathcal{A}_j\ne \emptyset$ for each $j\in [t]$, then we say that $\mathcal{A}_1, \mathcal{A}_2,\dots, \mathcal{A}_t$ are non-empty pairwise cross-intersecting. Let $\mathcal{A}_1\subset{[n]\choose k_1}, \mathcal{A}_2\subset{[n]\choose k_2}, \dots, \mathcal{A}_t\subset{[n]\choose k_t}$ be non-empty pairwise cross-intersecting families with $t\geq 2$, $k_1\geq k_2\geq \cdots \geq k_t$, $n\ge k_1+k_2$ and  $d_1, d_2, \dots, d_t$ be positive numbers. In this paper, we give a sharp upper bound of
$\sum_{j=1}^td_j|\mathcal{A}_j|$ and characterize  the families $\mathcal{A}_1, \mathcal{A}_2,\dots, \mathcal{A}_t$ attaining the upper bound. Our results unifies results of Frankl and Tokushige [J. Combin. Theory Ser. A 61 (1992)], Shi, Frankl and Qian [Combinatorica 42 (2022)], Huang and Peng \cite{huangpeng},  and Zhang-Feng \cite{ZF2023}.
Furthermore, our result can be  applied in the treatment for some $n<k_1+k_2$ while all previous known results do not have such an application. In the proof, a result of Kruskal-Katona  is applied to allow us to consider only families $\mathcal{A}_i$ whose elements are the first $|\mathcal{A}_i|$ elements in lexicographic order. We  bound $\sum_{i=1}^t{|\mathcal{A}_i|}$ by a single variable function $g(R)$, where $R$ is the last element of $\mathcal{A}_1$ in lexicographic order. One crucial and challenge part is  to verify that $-g(R)$ has unimodality. We think that the unimodality of functions in this paper are interesting in their own, in addition to the extremal result.
\end{abstract}



\section{Introduction}
Let $[n]=\{1, 2, \dots, n\}$.  For $0\leq k \leq n$, let ${[n]\choose k}$ denote the family of all $k$-subsets of $[n]$. A family $\mathcal{A}$ is $k$-uniform if $\mathcal{A}\subset {[n]\choose k}$. A family $\mathcal{A}$ is  intersecting if $A\cap B\ne \emptyset$ for any $A$ and $B\in \mathcal{A}$. Many researches in extremal set theory are inspired by the foundational  result of Erd\H{o}s--Ko--Rado \cite{EKR1961} showing that a maximum $k$-uniform intersecting family is a full star.  This theorem of Erd\H{o}s--Ko--Rado has many interesting generalizations. Two families $\mathcal{A}$ and $\mathcal{B}$ are cross-intersecting if $A\cap B\ne \emptyset$ for any $A\in \mathcal{A}$ and $B\in \mathcal{B}$.
We call $t$ $(t\geq 2)$ families  $\mathcal{A}_1, \mathcal{A}_2,\dots, \mathcal{A}_t$ pairwise cross-intersecting families if $\mathcal{A}_i$ and $\mathcal{A}_j$ are cross-intersecting when $1\le i<j \le t$. Additionally, if $\mathcal{A}_j\ne \emptyset$ for each $j\in [t]$, then we say that $\mathcal{A}_1, \mathcal{A}_2,\dots, \mathcal{A}_t$ are non-empty pairwise cross-intersecting.
The following result was proved by Hilton.
\begin{theorem}[Hilton, \cite{H}]
Let $n, k$ and $t$ be positive integers with $n\geq 2k$ and $t\geq 2$. If $\mathcal{A}_1, \mathcal{A}_2,\dots, \mathcal{A}_t\subset {[n]\choose k}$ are pairwise cross-intersecting, then
\begin{align*}
\sum_{i=1}^t|\mathcal{A}_i| \leq
\begin{cases} {n \choose k}, & \text{if $t\le \frac{n}{k}$}; \\
t{n-1 \choose k-1}, & \text{if $t\ge \frac{n}{k}$},
\end{cases}
\end{align*}
and the  bound is tight. If $|\mathcal{A}_1|\geq |\mathcal{A}_2|\geq \cdots \geq |\mathcal{A}_t|$, $n\ne 2k$ when $t=2$, and the equality holds, then either
$\mathcal{A}_1={[n]\choose k}$, $\mathcal{A}_2=\cdots=\mathcal{A}_t=\emptyset$ and $t\le \frac{n}{k}$, or  $\mathcal{A}_1=\mathcal{A}_2=\cdots=\mathcal{A}_t=\{F\in {[n]\choose k}: x\in F, \ {\rm where } \ x\in [n] \}$ and $t\ge \frac{n}{k}$.
\end{theorem}
For non-empty situation, Hilton and Milner gave the following result.
\begin{theorem}[Hilton--Milner, \cite{HM1967}]\label{HM}
Let $n$ and $k$ be positive integers with $n\geq 2k$ and $\mathcal{A}, \mathcal{B}\subset {[n]\choose k}$. If $\mathcal{A}$ and $\mathcal{B}$ are non-empty cross-intersecting, then
$$
|\mathcal{A}|+|\mathcal{B}|\leq {n\choose k} -{n-k\choose k}+1.
$$
\end{theorem}
The upper bound is achievable at $\mathcal{A}=\{[k]\}$ and $\mathcal{B}=\{F\in {[n]\choose k}: F\cap [k]\ne \emptyset\}$.
More generally, Frankl and Tokushige showed that
\begin{theorem}[Frankl-Tokushige, \cite{FT}]\label{FT1992}
Let $\mathcal{A}\subset {[n]\choose k}$ and $\mathcal{B}\subset {[n]\choose l}$ be non-empty cross-intersecting families with $n\geq k+l$ and $k\geq l$. Then
$$|\mathcal{A}|+|\mathcal{B}|\leq{n\choose k}-{n-l\choose k}+1.$$
\end{theorem}
The upper bound is achievable at $\mathcal{A}=\{[l]\}$ and $\mathcal{B}= \{F\in {[n]\choose k}: F\cap [l]\ne \emptyset\}$.
Borg and Feghali \cite{BF} got the analogous maximum sum problem for the case when
$\mathcal{A}\subset {[n]\choose \leq r}$ and $\mathcal{B}\subset {[n]\choose \leq s}$.
\begin{theorem}[Borg--Feghali, \cite{BF}]
Let $n\geq 1, 1\leq r\leq s, \mathcal{A}\subset  {[n]\choose \leq r}$ and $\mathcal{B}\subset { [n]\choose \leq s}$. If $\mathcal{A}$ and $\mathcal{B}$ are non-empty cross-intersecting, then
 $$|\mathcal{A}|+|\mathcal{B}|\leq 1+\sum_{i=1}^s\left( {n\choose i}-{n-r\choose i}\right),$$
 and equality holds if $\mathcal{A}=\{[r]\}$ and $\mathcal{B}=\{ B\in { [n] \choose \leq s} : B\cap [r]\ne \emptyset\}.$
\end{theorem}

Recently, Shi, Frankl and Qian proved the following result.
\begin{theorem}[Shi--Frankl--Qian, \cite{SFQ2020}]\label{SFQ}
Let $n, k, l, r$ be integers with $n\geq k+l, l\geq r\geq 1$, $c$ be a positive constant and $\mathcal{A}\subset {[n]\choose k}, \mathcal{B}\subset {[n]\choose l}$. If $\mathcal{A}$ and $\mathcal{B}$ are cross-intersecting and ${n-r\choose l-r}\leq |\mathcal{B}|\leq {n-1\choose l-1}$, then
$$|\mathcal{A}|+c|\mathcal{B}|\leq \textup{max} \left\{{n\choose k}-{n-r\choose k}+c{n-r\choose l-r}, \,{n-1\choose k-1}+c{n-1\choose l-1}\right\}$$
and the upper bound is sharp.
\end{theorem}
In \cite{SFQ2020}, Shi, Frankl and Qian also gave the families attaining the equality.
Setting $c=t-1$ in Theorem \ref{SFQ}, they got the following interesting corollary which is a generalization of Theorem \ref{HM}.

\begin{theorem}[Shi--Frankl--Qian, \cite{SFQ2020}]\label{sfq2020}
Let $n$ and $k$ be positive integers with $n\geq 2k$ and $t\geq 2$. If $\mathcal{A}_1, \mathcal{A}_2, \dots, \mathcal{A}_t\subset {[n]\choose k}$ are non-empty pairwise cross-intersecting families, then
$$
\sum_{i=1}^t{|\mathcal{A}_i|}\leq \textup{max} \left\{{n\choose k}-{n-k\choose k}+t-1, t{n-1\choose k-1}\right\},
$$
and the upper bound is sharp.
\end{theorem}

Furthermore, Shi, Frankl and Qian \cite{SFQ2020} proposed the following problem.

\begin{problem}(Shi--Frankl--Qian, \cite{SFQ2020})\label{1}
Let $\mathcal{A}_1\subset{[n]\choose k_1}, \mathcal{A}_2\subset{[n]\choose k_2}, \dots, \mathcal{A}_t\subset{[n]\choose k_t}$ be non-empty pairwise cross-intersecting families with $t\geq 2$, $k_1\geq k_2\geq \cdots \geq k_t$, and $n\geq k_1+k_2$. Is it true that
$$
\sum_{i=1}^t{|\mathcal{A}_i|}\leq \max \left\{{n\choose k_1}-{n-k_t\choose k_1}+\sum_{i=2}^t{{n-k_t\choose k_i-k_t}}, \sum_{i=1}^t{n-1\choose k_i-1}\right\}?
$$

\end{problem}

 As mentioned above that Shi, Frankl and Qian \cite{SFQ2020} obtained a positive answer to the above problem for the special case that $k_1=k_2=\cdots=k_t$ (Theorem \ref{sfq2020}) by taking $c=t-1$ in the result of the maximum value of  $|\mathcal{A}|+c|\mathcal{B}|$ for two non-empty cross-intersecting families $\mathcal{A}$ and $\mathcal{B}$ (Theorem \ref{SFQ}). In \cite{huangpeng} and \cite{ZF2023}, both groups  gave a positive answer to the above problem by different methods.
 \begin{theorem}\cite{huangpeng, ZF2023}\label{hpzf}
Let $\mathcal{A}_1\subset{[n]\choose k_1}, \mathcal{A}_2\subset{[n]\choose k_2}, \dots, \mathcal{A}_t\subset{[n]\choose k_t}$ be non-empty pairwise cross-intersecting families with $t\geq 2$, $k_1\geq k_2\geq \cdots \geq k_t$, and $n\geq k_1+k_2$. Then
$$
\sum_{i=1}^t{|\mathcal{A}_i|}\leq \textup{max} \left\{{n\choose k_1}-{n-k_t\choose k_1}+\sum_{i=2}^t{{n-k_t\choose k_i-k_t}}, \,\,\sum_{i=1}^t{n-1\choose k_i-1}\right\},
$$
and the bound is sharp.
\end{theorem}
The bound is sharp by the following non-empty pairwise cross-intersecting families:
Let $\mathcal{A}_j=\{F\in{[n]\choose k_j}: 1\in F\}$ for each $j\in [t]$, then $\mathcal{A}_j$, $1\le j\le t$ are non-empty pairwise cross-intersecting families.  Let $\mathcal{A}_1=\{F\in {[n]\choose k_i}: F\cap [k_t] \ne\emptyset\}$ and $\mathcal{A}_j=\{F\in {[n]\choose k_j}: [k_t]\subset F\}$ for each $j\in [t]\setminus \{1\}$, then $\mathcal{A}_j$, $1\le j\le t$ are non-empty pairwise cross-intersecting families. In \cite{huangpeng} and \cite{ZF2023}, the authors also gave the families attaining the equality.

 In this paper, we give more general results. Our main results are as follows.

For a family ${\mathcal{A}}$ of subsets $[n]$,  let $\overline{\mathcal{A}}=\{[n]\setminus A: A\in \mathcal{A}\}$.

\begin{theorem}\label{2}
Let $n$, $t\geq 2$, $k_1, k_2, \dots, k_t$ be positive integers and $d_1, d_2, \dots, d_t$ be positive numbers. Let
$\mathcal{A}_1\subset{[n]\choose k_1}, \mathcal{A}_2\subset{[n]\choose k_2}, \dots, \mathcal{A}_t\subset{[n]\choose k_t}$ be non-empty pairwise cross-intersecting families with $|\mathcal{A}_i|\geq {n-1\choose k_i-1}$ for some $i\in [t]$. Let $m_i$ be the minimum integer among $k_j$, where $j\in [t]\setminus \{i\}$. If $n\geq k_i+k_j$ for all $j\in [t]\setminus \{i\}$, then
$$\sum_{1=j}^td_j|\mathcal{A}_j|\leq \max  \left\{d_i{n\choose k_i}-d_i{n-m_i\choose k_i}+\sum_{j\ne i}d_j{n-m_i\choose k_j-m_i}, \,\,\sum_{j=1}^td_j{n-1\choose k_j-1}\right\},$$
the equality holds if and only if one of the following holds.\\
(1) If $d_i{n\choose k_i}-d_i{n-k_t\choose k_i}+\sum_{j\ne i}^td_j{n-m_i\choose k_j-m_i}\geq \sum_{j=1}^td_j{n-1\choose k_j-1}$, then  there is some $m_i$-element set $T\subset [n]$ such that $\mathcal{A}_i=\{F\in {[n]\choose k_i}: F\cap T\ne\emptyset\}$ and $\mathcal{A}_j=\{F\in {[n]\choose k_j}: T\subset F\}$ for each $j\in [t]\setminus \{i\}$.\\
(2) If
$d_i{n\choose k_i}-d_i{n-k_t\choose k_i}+\sum_{j\ne i}^td_j{n-m_i\choose k_j-m_i}\leq \sum_{j=1}^td_j{n-1\choose k_j-1}$,
then there is some $a\in [n]$ such that $\mathcal{A}_j=\{F\in{[n]\choose k_j}: a\in F\}$ for each $j\in [t]$. \\
(3) If $t=2$ and $n=k_i+k_{3-i}$.  If $d_i\leq d_{3-i}$,  then $\mathcal{A}_{3-i}\subseteq {[n]\choose k_{3-i}}$ with $|\mathcal{A}_{3-i}|={n-1\choose k_{3-i}-1}$ and $\mathcal{A}_i={[n]\choose k_i}\setminus \overline{\mathcal{A}_{3-i}}$. \\
(4)
If $n=k_i+k_j$ holds for every $j\in [t]\setminus \{i\}$ and $\sum_{j\ne i}d_j=d_i$, then  $\mathcal{A}_j=\mathcal{A}$ for all $j\in [t]\setminus \{i\}$, where $\mathcal{A}\subseteq {[n]\choose k}$ is an intersecting family with size $|\mathcal{A}|={n-1\choose k-1}$, and $\mathcal{A}_i={[n]\choose k_i}\setminus \overline{\mathcal{A}}$.


\end{theorem}

Let $$\lambda_i=d_i{n\choose k_i}-d_i{n-m_i\choose k_i}+\sum_{j\ne i}d_j{n-m_i\choose k_j-m_i}.$$

\begin{theorem}\label{2+}
Let $n$,  $t\geq 2$, $k_1, k_2, \dots, k_t$ be positive integers with
$k_1\geq k_2\geq \dots\geq k_t$ and $n\geq k_1+k_2$.
 Let $d_1, d_2, \dots, d_t$ be positive numbers. Let
$\mathcal{A}_1\subset{[n]\choose k_1}, \mathcal{A}_2\subset{[n]\choose k_2}, \dots, \mathcal{A}_t\subset{[n]\choose k_t}$ be non-empty pairwise cross-intersecting families.  Let $m_i$ be the minimum integer among $k_j$, where $j\in [t]\setminus \{i\}$, and  $$\lambda_i=d_i{n\choose k_i}-d_i{n-m_i\choose k_i}+\sum_{j\ne i}d_j{n-m_i\choose k_j-m_i}.$$ Then
$$\sum_{1=j}^td_j|\mathcal{A}_j|\leq \max  \left\{\lambda_1, \dots, \lambda_t,\,\sum_{j=1}^td_j{n-1\choose k_j-1}\right\},$$
and the bound is sharp.
\end{theorem}
The bound is sharp by the following non-empty pairwise cross-intersecting families:
Let $\mathcal{A}_j=\{F\in{[n]\choose k_j}: 1\in F\}$ for each $j\in [t]$, then $\mathcal{A}_j$, $1\le j\le t$ are non-empty pairwise cross-intersecting families.  Let $\mathcal{A}_i=\{F\in {[n]\choose k_i}: F\cap [m_i] \ne\emptyset\}$ and $\mathcal{A}_j=\{F\in {[n]\choose k_j}: [m_i]\subset F\}$ for each $j\in [t]\setminus \{i\}$, then $\mathcal{A}_j$, $1\le j\le t$ are non-empty pairwise cross-intersecting families.

\begin{theorem}\label{2++}
Let $t\geq 2$, $k_1, k_2, \dots, k_t$ be positive integers with
$k_1> k_2> \dots> k_t$ and $n\geq k_1+k_2$. Let $d_1, d_2, \dots, d_t$ be positive numbers with $d_1\geq d_2\geq \dots\geq d_t$. Let
$\mathcal{A}_1\subset{[n]\choose k_1}, \mathcal{A}_2\subset{[n]\choose k_2}, \dots, \mathcal{A}_t\subset{[n]\choose k_t}$ be non-empty pairwise cross-intersecting families.  Then
$$\sum_{j=1}^td_j|\mathcal{A}_j|\leq \max  \left\{d_1{n\choose k_1}-d_1{n-k_t\choose k_1}+\sum_{j=2}^td_j{n-k_t\choose k_j-k_t}, \,\,\sum_{j=1}^td_j{n-1\choose k_j-1}\right\},$$
and the bound is sharp.
\end{theorem}
The bound is sharp by the following non-empty pairwise cross-intersecting families:
Let $\mathcal{A}_j=\{F\in{[n]\choose k_j}: 1\in F\}$ for each $j\in [t]$, then $\mathcal{A}_j$, $1\le j\le t$ are non-empty pairwise cross-intersecting families.  Let $\mathcal{A}_1=\{F\in {[n]\choose k_1}: F\cap [k_t] \ne\emptyset\}$ and $\mathcal{A}_j=\{F\in {[n]\choose k_j}: [k_t]\subset F\}$ for each $j\in [t]\setminus \{1\}$, then $\mathcal{A}_j$, $1\le j\le t$ are non-empty pairwise cross-intersecting families.

 One application of our main results is for `mixing' cases  while other previous known results (for example, Theorem \ref{hpzf}) do not have such an application, let us explain.  Note that if  $n< k+l$, then all families contained in ${[n]\choose k}$ and all families contained in ${[n]\choose l}$ are cross-intersecting. In this situation, we say that families contained in ${[n]\choose k}$ and  families contained in ${[n]\choose l}$ are automatically cross-intersecting.
 The condition that $n\geq k_1+k_2$ in Theorems \ref{2+} and \ref{2++} is to guarantee that no pair of families contained in ${[n]\choose k_i}$ and ${[n]\choose k_j}$ are automatically cross-intersecting for any $1\le i<j\le t$. It is also interesting to consider the maximum value of $\sum_{1=j}^t |\mathcal{A}_j|$ for non-empty  pairwise cross-intersecting families when families contained in ${[n]\choose k_i}$ and ${[n]\choose k_j}$ are automatically cross-intersecting for some but not all $i$ and $j$.
 For example, if $k_1+k_3\le n<k_1+k_2$, then all  families contained in  ${[n]\choose k_1}$ and all  families contained in  ${[n]\choose k_2}$   are automatically cross-intersecting, on the other hand, families contained in ${[n]\choose k_i}$  and  families contained in ${[n]\choose k_j}$ are not  automatically cross-intersecting for $\{i, j\}\neq \{1, 2\}$. Let us use this as an example to explain the application of Theorem \ref{2++} to this type of  `mixing' cases (some of them but not all of them are automatically cross-intersecting).  Let us be precise: Let $t\geq 2$, $k_1, k_2, \dots, k_t$ be positive integers with
$k=k_1= k_2\geq \dots\geq k_t$ and $k_1+k_3\le n<k_1+k_2$. Let $\mathcal{A}_1\subset{[n]\choose k_1}, \mathcal{A}_2\subset{[n]\choose k_2}, \dots, \mathcal{A}_t\subset{[n]\choose k_t}$ be non-empty pairwise cross-intersecting families. What is the maximum value of $\sum_{j=1}^t |\mathcal{A}_j|$? Taking $d_1=1$ and $d_i={1 \over 2}$ for $i\in [3, t]$, and applying Theorem \ref{2++} to $\mathcal{A}_1, \mathcal{A}_3, \mathcal{A}_4, \cdots \mathcal{A}_t$, we have
\begin{equation}\label{appeq1}
 |\mathcal{A}_1|+{1 \over 2}\sum_{j=3}^t|\mathcal{A}_j|\leq \max  \left\{{n\choose k}-{n-k_t\choose k}+{1 \over 2}\sum_{j=3}^t {n-k_t\choose k_j-k_t}, \,\,{n-1\choose k-1}+{1 \over 2}\sum_{j=3}^t{n-1\choose k_j-1}\right\}.
 \end{equation}
Taking $d_2=1$ and $d_i={1 \over 2}$ for $i\in [3, t]$, and applying Theorem \ref{2++} to $\mathcal{A}_2, \mathcal{A}_3, \mathcal{A}_4, \cdots \mathcal{A}_t$, we have
\begin{equation}\label{appeq2}
 |\mathcal{A}_2|+{1 \over 2}\sum_{j=3}^t|\mathcal{A}_j|\leq \max  \left\{{n\choose k}-{n-k_t\choose k}+{1 \over 2}\sum_{j=3}^t {n-k_t\choose k_j-k_t}, \,\,{n-1\choose k-1}+{1 \over 2}\sum_{j=3}^t{n-1\choose k_j-1}\right\}.
 \end{equation}
 Combining inequalities (\ref{appeq1}) and (\ref{appeq2}), we have
 $$\sum_{1=j}^t |\mathcal{A}_j|\le \max  \left\{2({n\choose k}-{n-k_t\choose k})+\sum_{j=3}^t {n-k_t\choose k_j-k_t}, \,\,2{n-1\choose k-1}+\sum_{j=3}^t{n-1\choose k_j-1}\right\}.$$
 In summary, we have the following corollary.
 \begin{corollary}
 Let $t\geq 2$, $k_1, k_2, \dots, k_t$ be positive integers with
$k=k_1= k_2\geq \dots\geq k_t$ and $k_1+k_3\le n<k_1+k_2$. Let $\mathcal{A}_1\subset{[n]\choose k_1}, \mathcal{A}_2\subset{[n]\choose k_2}, \dots, \mathcal{A}_t\subset{[n]\choose k_t}$ be non-empty pairwise cross-intersecting families. Then
 $$\sum_{1=j}^t |\mathcal{A}_j|\le \max  \left\{2({n\choose k}-{n-k_t\choose k})+\sum_{j=3}^t {n-k_t\choose k_j-k_t}, \,\,2{n-1\choose k-1}+\sum_{j=3}^t{n-1\choose k_j-1}\right\},$$
and the bound is sharp.
 \end{corollary}
If we relax $n$ to a more general condition $n<k_1+k_2$ and relax the condition $k_1= k_2$ to the natural condition $k_1\le k_2$, then it is far from determining the maximum value of $\sum_{1=j}^t |\mathcal{A}_j|$ to simply apply Theorem \ref{2++}. Developing the method in this paper and overcoming more challenges, we can go further and answer the general question: what is $\sum_{j=1}^t |\mathcal{A}_j|$ for any `mixing' case. Indeed, our method is a basis,  there are more ingredients in the proof, we will reveal it in another manuscript.

In both \cite{SFQ2020} and our paper, a result of Kruskal-Katona (Theorem \ref{3}) is applied to allow us to consider only families $\mathcal{A}_i$ whose members are the first $|\mathcal{A}_i|$ members in lexicographic order.
We analyze the relationship between  $\sum_{j=1}^td_j|\mathcal{A}_j|$ and the last member (in the lexicographic order) of  $\mathcal{A}_i$. Let $R$ be the last member of  $\mathcal{A}_i$, we will bound $\sum_{j=1}^td_j|\mathcal{A}_j|$ by a function $f_i(R)$. In order to do this, we  introduce new  concepts `$c$-sequential' and `down-up family', and show four types of  `local unimodality' of  $-f_i(R)$  in Lemmas \ref{clm3}, \ref{clm28},  \ref{clm29} and \ref{clm30}.

\section{ Proof for Theorems \ref{2}, \ref{2+} and \ref{2++}}
 When we write a set $A=\{a_1, a_2, \ldots, a_s\}\subset [n]$, we always assume that $a_1<a_2<\ldots<a_s$ throughout the paper.  Let us introduce the lexicographic (lex for short) order of subsets of positive integers. Let $A$ and $B$ be finite subsets of the set of positive integers $\mathbb{Z}_{>0}$. We say that $A\prec B$ if either $A\supset B$ or $\min(A\setminus B) < \min(B\setminus A)$. In particular, $A\prec A$. Let $\mathcal{L}([n], r, k)$ denote the first $r$ subsets in ${[n]\choose k}$ in the lex order. Given a set $R$, we denote $\mathcal{L}([n], R, k)=:\{F\in {[n]\choose k}: F\prec R\}$.  Let $\mathcal{F}\subset {[n]\choose k}$ be a family, we say $\mathcal{F}$ is {\it L-initial} if $\mathcal{F}=\mathcal{L}([n], |\mathcal{F}|, k)$.

The well-known Kruskal-Katona theorem \cite{KK1, KK2} will play an important role in our discussion, an equivalent formulation of which was given in \cite{KK3, KK4} as follows.

\begin{theorem}[Kruskal-Katona, \cite{KK1, KK2}]\label{3}
For $\mathcal{A}\subset {[n]\choose k}$ and $\mathcal{B}\subset {[n]\choose l}$, if $\mathcal{A}$ and $\mathcal{B}$ are cross-intersecting, then $\mathcal{L}([n], |\mathcal{A}|, k)$ and $\mathcal{L}([n], |\mathcal{B}|, l)$ are cross-intersecting as well.
\end{theorem}

\subsection{Sketch of the proof of Theorem \ref{2}}
We first consider the quantitative part of Theorem \ref{2}.  By Theorem \ref{3},   we may assume that $\mathcal{A}_i$ is L-initial, that is, $\mathcal{A}_i=\mathcal{L}([n], |\mathcal{A}_i|, k_i)$ for each $i\in [t]$.
In this section, we give an outline of the proof and leave the proofs of some crucial propositions and  lemmas to Subsection \ref{subprooflemma} and Section \ref{sec3}.
Recall that
\begin{equation}\label{m}
m_i=\min_{j\ne i}k_j.
\end{equation}
Note that $\mathcal{A}_1, \dots, \mathcal{A}_t$ are pairwise cross intersecting and L-initial, 
we have the following proposition.
\begin{proposition}\label{coro7}
$|\mathcal{A}_i|\leq{n-1\choose k_i-1}+\cdots+{n-m_i\choose k_i-1}$.
\end{proposition}

One important ingredient of the proof is to bound   $\sum_{i=1}^t{|\mathcal{A}_i|}$ by a function of  the last element of $\mathcal{A}_i$. Let us list the set of the last elements of all possible $\mathcal{A}_i$.

Let $Z={n-2\choose k_i-1}+\cdots+{n-m_i\choose k_i-1}, R_0=\{1, n-k_i+2, n-k_i+3, \dots, n\}, R_1=\{2, 3, \cdots, k_i+1\}$, $R_Z=\{m_i, n-k_i+2, n-k_i+3, \dots, n\}$ and   $R_0\precneqq R_1\precneqq \cdots \precneqq R_Z$ in lex order with each $|R_j|=k_i$ for $j\in [Z]$.
We denote
\begin{align}  \label{eq-R}
\mathbb{R}_i=:\{R_0, R_1, \dots, R_Z\}.
\end{align}

By Proposition \ref{coro7}, we have ${n-1\choose k_i-1}\leq |\mathcal{A}_i|\leq {n-1\choose k_i-1}+Z$. Since $\mathcal{A}_i$ is L-initial, we have the following remark.

\begin{remark}\label{clm12}
Let $0\leq r\leq Z$. If $|\mathcal{A}_i|={n-1\choose k_i-1}+r$, then $\mathcal{A}_i=\mathcal{L}([n], R_r, k_i)$.
\end{remark}

 Let $R$  be the last element  of  $\mathcal{A}_i$ (we call $R$ the ID of $\mathcal{A}_i$), clearly $R\in \mathbb{R}_i$.
 We will bound   $\sum_{i=1}^t{|\mathcal{A}_i|}$ by a function of  $R$. In order to do this, we will extend a  result of  Frankl-Kupavskii (Proposition \ref{prop15}).

\begin{definition}
We say that $A$ and $B$ strongly intersect at their last element $q$ if $A\cap B=\{q\}$ and $A\cup B=[q]$. We also say $A$  is $B$'s  partner.
\end{definition}

We have the following observation.

\begin{remark}\label{coro11}
Let $k, l, n\in \mathbb{Z}_{>0}, n\geq k+l, R\subset [n], |R|=k$ and $\max R=n$. Let $p$ be the last element of $R$ not continuing to $n$ and $R'=R\cap [p]$. Let $T$ and $T'$ be the partners of $R$ and $R'$ respectively. Then $\mathcal{L}([n], R, k)=\mathcal{L}([n], R', k)$ and $\mathcal{L}([n], T, l)=\mathcal{L}([n], T', l)$.
\end{remark}


\begin{definition}
Let $t\geq 2$. We say that $\mathcal{F}_1\subset{[n]\choose l_1}, \mathcal{F}_2\subset{[n]\choose l_2}, \dots, \mathcal{F}_t\subset{[n]\choose l_t}$ are maximal pairwise cross-intersecting if whenever $\mathcal{F}'_1\subset{[n]\choose l_1}, \mathcal{F}'_2\subset{[n]\choose l_2}, \dots, \mathcal{F}'_t\subset{[n]\choose l_t}$ are pairwise cross-intersecting with $\mathcal{F}'_1\supset \mathcal{F}_1, \dots, \mathcal{F}'_t\supset \mathcal{F}_t$, then $\mathcal{F}_1=\mathcal{F}'_1, \dots, \mathcal{F}_t=\mathcal{F}'_t$.
\end{definition}

\begin{proposition}[Frankl-Kupavskii \cite{FK2018}]\label{prop15}
Let $a, b\in \mathbb{Z}_{>0}, a+b\leq n$. Let $P$ and $Q$ be non-empty subsets of $[n]$ with $|P|\leq a$ and $|Q|\leq b$. If $Q$ is the partner of $P$, then $\mathcal{L}([n], P, a)$ and $\mathcal{L}([n], Q, b)$ are maximal cross-intersecting families.
\end{proposition}

This result cannot be applied to our situation directly. We  get rid of the condition $|Q|\leq b$ in Proposition \ref{prop15} and showed the following result.

\begin{proposition}\label{coro16}
Let $a, b, n\in \mathbb{Z}_{>0}$ and $a+b\leq n$. For $P\subset [n]$ with $|P|\leq a$, let $Q$ be the partner of $P$. Then $\mathcal{L}([n], Q, b)$ is the maximum L-initial $b$-uniform family that is cross-intersecting to $\mathcal{L}([n], P, a)$. Moreover, $\mathcal{L}([n], Q, b)\ne\emptyset$ if and only if $\min P\leq b$.
\end{proposition}

Before to prove the above proposition, we give the following observation.

\begin{remark}\label{coro10}
Let $k, n\in \mathbb{Z}_{>0}$ and $A=\{a_1, a_2, \dots, a_{|A|}\}\subset [n]$ with $|A|>k$. Let $j=\max\{q:q\in [a_k]\setminus A\}$ and $A'=(A\cap [j])\cup\{j\}$. Then $\mathcal{L}([n], A, k)=\mathcal{L}([n], A', k)$.
\end{remark}

\begin{proof}[Proof of Proposition \ref{coro16}]
Let $P_0$ be the last element of $\mathcal{L}([n], P, a)$. If $|P|=a$, then  $P_0=P$  and if $|P|<a$, then $P_0=P\cup [n-a+|P|+1, n]$.
So $\min P_0=\min P$. Then $[b]\cap P_0=\emptyset$ if and only if $\min P>b$, this implies that $\mathcal{L}([n], Q, b)\ne\emptyset$ if and only if $\min P\leq b$. As desired. So we may assume $\min P\leq b$.

By Proposition \ref{prop15}, we only need to consider the case that $|Q|>b$. We first show that $\mathcal{L}([n], Q, b)$ and $\mathcal{L}([n], P, a)$ are cross-intersecting. For any $F\in \mathcal{L}([n], Q, b)$, we have $\min F\setminus Q <\min Q\setminus F$. Let $z_1=\min F\setminus Q$. Then $z_1\in P$ since $P$ is $Q$'s partner and $z_1<\min Q\setminus F \leq \max Q=\max P$. This implies that $F\cap P_0\ne \emptyset$. Let $P'\precneqq P_0$ with $|P'|=a$. If $P\subseteq P'$, then $F\cap P'\ne \emptyset$ since $z_1\in F\cap P'$. So we may assume $P\not\subseteq P'$. This implies
 $\min P'\setminus P<\min P\setminus P'$. Let $z_2=\min P'\setminus P$, then $z_2\in Q$ since $Q$ is $P$'s partner and $z_2< \min P\setminus P'\leq \max P=\max Q$. If $z_2\in F$, then $F\cap P'\ne \emptyset$. Suppose $z_2\not\in F$. If $z_1\in P'$, then $F\cap P'\ne \emptyset$. So assume that $z_1\not\in P'$. Since $z_1\in P$, we get $z_2=\min P'\setminus P<\min P\setminus P'\leq z_1$. However, $z_2\in Q, z_2\not\in F$, so $z_2\geq \min Q\setminus F>\min F\setminus Q=z_1$, a contradiction. We have proved that $\mathcal{L}([n], P, a)$ and $\mathcal{L}([n], Q, b)$ are cross-intersecting.

Next we show that $\mathcal{L}([n], Q, b)$ is the maximal L-initial $b$-uniform family that is cross-intersecting to $\mathcal{L}([n], P, a)$.
Let $Q_b$ be the $b$-th element of $Q$. Since $\min P\leq b$, then $Q_b>b$ and $[Q_b]\setminus Q\ne \emptyset$. Let $y=\max \{q: q\in [Q_b]\setminus Q\}$ and $Q'=(Q\cap [y])\cup \{y\}$. Then $|Q'|\leq b$. By Remark \ref{coro10}, $\mathcal{L}([n], Q', b)=\mathcal{L}([n], Q, b)$.
Suppose that $\mathcal{G}$ is another $b$-uniform L-initial family cross-intersecting with $\mathcal{L}([n], P, a)$ and $|\mathcal{G}|>|\mathcal{L}([n], Q', b)|$. Then $\mathcal{G}\supsetneqq \mathcal{L}([n], Q', b)$. Let $H$ be the last set in $\mathcal{L}([n], Q', b)$ and $G$ be the first set in $\mathcal{G}\setminus \mathcal{L}([n], Q', b)$.
Clearly $y=\max Q'< n$.
Let $|Q'|=p$. We have the following two cases.

Case (i) $|Q'|=b$. In this case $H=Q'$. Then $G=(Q'\setminus\{y\})\cup\{y+1\}$. Since $y\not\in Q$ and $y< \max Q=\max P$, $y\in P$. By our  definition of $y$, $y+1\in Q$. And $y+1\not\in P$, otherwise $|Q|=b$, also a contradiction. However, by the definition of $Q'$, we have $Q'\cap P=\{y\}$, so $G\cap P=\emptyset$, therefore, $G\cap P_0=\emptyset$, a contradiction again.

Case (ii) $|Q'|<b$. In this case $H=Q'\cup\{n-b+p+1, \dots, n\}$ and $G=(Q'\setminus\{y\})\cup\{y+1, y+2, \dots, y+b-p+1\}.$ Moreover, by the definitions of $Q_b$ and $y$, we can see that $y+b-p+1=Q_b$ and $\{y+1, y+2, \dots, y+b-p+1\}\subset Q$. Since $Q$ is the partner of $P$ and $Q_b<\max Q=\max P$, $\{y+1, y+2, \dots, y+b-p+1\}\cap P=\emptyset$. Recall that  $Q'\cap P=\{y\}$, so $G\cap P=\emptyset$, therefore, $G\cap P_0=\emptyset$, a contradiction.
So we have shown that $\mathcal{L}([n], Q', b)$, the same as  $\mathcal{L}([n], Q, b)$ (see Remark \ref{coro10}) is the maximum $b$-uniform L-initial family that is cross-intersecting to $\mathcal{L}([n], P, a)$, as desired.
\end{proof}

We give a  formula to calculate the size of an $L-$initial family.
\begin{proposition}\label{prop9}
Let $k, l, n$ be positive integers. Let $A=\{a_1, a_2, \dots, a_{s_a}\}\subset[n]$ and $B=\{b_1, b_2, \dots, b_{s_b}\}$ be $A$'s partner. Then
\begin{align}\label{eq1}
&|\mathcal{L}([n], A, k)|={n-b_1\choose k-b_1}+{n-b_2\choose k-b_2+1}+\cdots+{n-b_{s_b}\choose k-b_{s_b}+s_b-1},\\\label{eq2}
&|\mathcal{L}([n], B, l)|={n-a_1\choose l-a_1}+{n-a_2\choose l-a_2+1}+\cdots+{n-a_{s_a}\choose l-a_{s_a}+s_a-1}.
\end{align}
\end{proposition}

\begin{proof}
We give the proof of (\ref{eq1}) only,  since the proof of (\ref{eq2}) is similar to (\ref{eq1}).
W.l.o.g., assume $b_1=1$.  Let $D_1:=\{1, \dots, a_1-1\}, D_j:=\{a_{j-1}+1, \dots, a_j-1\}$ for $j\in [2, s_a-1]$, and $D_{s_a}:=\{a_{s_{a-1}}+1, \dots, a_{s_a}\}$. Let $B:=\sqcup_{j=1}^{s_a}D_j$.
If $k<s_a$, then
\begin{align*}
\mathcal{L}([n], A, k)
&=\{F\in {[n]\choose k}: F\prec A\} \\
&=\{F\in {[n]\choose k}: F\cap D_1\ne \emptyset\} \sqcup \{F\in {[n]\choose k}: F\cap D_1=\emptyset, a_1\in F, F\cap D_2\ne \emptyset\}  \\
&\quad \sqcup \cdots \sqcup \{F\in {[n]\choose k}: F\cap D_j=\emptyset, a_j\in F\,\,\text{for}\,\, j\in[k-1], F\cap D_k\ne \emptyset\}.
\end{align*}
If $k\geq s_a$, then
\begin{align*}
\mathcal{L}([n], A, k)
&=\{F\in {[n]\choose k}: F\cap D_1\ne \emptyset\} \sqcup \{F\in {[n]\choose k}: F\cap D_1=\emptyset, a_1\in F, F\cap D_2\ne \emptyset\}  \\
&\quad \sqcup \cdots \sqcup \{F\in {[n]\choose k}: F\cap D_j=\emptyset, a_j\in F \,\,\text{for}\,\, j\in[{s_a}-1], F\cap D_{s_a}\ne \emptyset\}  \\
&\quad \sqcup\{F\in {[n]\choose k}: F\cap [a_{s_a}]=A\}.
\end{align*}
Thus,
\begin{align}\label{eq3}
|\mathcal{L}([n], A, k)|
&=\sum_{d=1}^{s_a}\sum_{j\in D_d}{n-j\choose k-d} \\ \label{eq4}
&={n-b_1\choose k-b_1}+{n-b_2\choose k-b_2+1}+\cdots+{n-b_{s_b}\choose k-b_{s_b}+s_b-1},
\end{align}
as desired.
\end{proof}

Combining Proposition \ref{coro16} and Proposition \ref{prop9}, we can bound $\sum_{j=1}^t d_j|\mathcal{A}_j|$ based on the ID of $\mathcal{A}_i$ as follows.

\begin{corollary}
Let $k_1, \dots, k_t$ be positive integers, $d_1, \dots, d_t$ be positive constants and $n\geq k_i+k_j$ for all $j\in [t]\setminus \{i\}$. Suppose that $\mathcal{A}_1\subseteq {[n]\choose k_1}$, $\mathcal{A}_2\subseteq {[n]\choose k_2}$, $\dots$, $\mathcal{A}_t\subseteq {[n]\choose k_t}$ are L-initial pairwise cross-intersecting families. Let $R=\{a_1, a_2, \dots, a_{k_i}\}$ be the ID of $\mathcal{A}_i$ with $\{1, n-k_i+2, \dots, n\}\prec R \prec \{k_t, n-k_i+2, \dots, n\}$ and $T=\{b_1, b_2, \dots, b_{s_b}\}$ be the partner of $R$.  Then
\begin{align}\label{new}
\sum_{j=1}^td_j|\mathcal{A}_j|
&\leq d_i|\mathcal{L}([n], R, k_i)|+\sum_{j\ne i}d_j|\mathcal{L}([n], T, k_j)|\nonumber\\
&=d_i\left[{n-b_1\choose k_i-b_1}+{n-b_2\choose k_i-b_2+1}+\cdots+{n-b_{s_b}\choose k_i-b_{s_b}+s_b-1}\right] \nonumber\\
& \quad +\sum_{j\ne i}d_j\left[{n-a_1\choose k_j-a_1}+{n-a_2\choose k_j-a_2+1}+\cdots+{n-a_{k_i}\choose k_j-a_{k_i}+k_i-1}\right]\nonumber\\
&=:f_i(R).
\end{align}
\end{corollary}
Thus, to show Theorem \ref{2}, it is sufficient to show that
$$f_i(R)\leq \max \left\{d_i{n\choose k_i}-d_i{n-m_i\choose k_i}+\sum_{j\ne i}d_j{n-m_i\choose k_j-m_i}, \,\,\sum_{j=1}^td_j{n-1\choose k_j-1}\right\}.
$$

Note that
\begin{eqnarray}\label{f1}
&&f_i(\{1, n-k_i+2, n-k_i+3, \ldots, n\})=\sum_{j=1}^td_j{n-1\choose k_j-1}, {\rm \ and \ correspondingly,} \nonumber \\
&&\vert \mathcal{A}_j\vert= {n-1\choose k_j-1} {\rm \ for \ each \ } j\in [t]
\end{eqnarray}
 in view of (\ref{new}). 
 And in view of (\ref{m})
\begin{eqnarray}\label{fm}
f_i(\{m_i\}\cup [n-k_i+2, n])&=& d_i{n\choose k_i}-d_i{n-m_i\choose m_i}+\sum_{j\ne i}d_j{n-k_t\choose k_i-k_t} {\rm \ and \ correspondingly,} \nonumber \\
\vert\mathcal{A}_i\vert&=&{n\choose k_i}-{n-m_i\choose k_i},  {\rm \ and \ } \vert\mathcal{A}_j\vert={n-m_i\choose k_j-m_i} {\rm \ if \ } j\in [t]\setminus\{i\}
\end{eqnarray}
in view of (\ref{new}).

Hence,  to show Theorem \ref{2}, it is sufficient to show that $f_i(R)\le \max\{f_i(\{1, n-k_i+2, n-k_i+3, \ldots, n \}), f_i(\{m_i, n-k_i+2, n-k_i+3, \ldots, n \})\}$.
For this purpose, we will introduce the concept `$c$-sequential' and show some `local convexity' of  $f_i(R)$.

Let $\mathcal{A}\subset {[n]\choose k}$ be a family and $c\in [k]$. We say that $\mathcal{A}$ is $c$-{\it sequential} if there are $A\subset [n]$ with $|A|=k-c$ and $a\geq \max A$ (For a set $A\subset [n]$,  denote $\max A=\max \{a: a\in A\}$ and $\min A=\min \{a: a\in A\}$) such that $\mathcal{A}=\{A\sqcup \{a+1, \dots, a+c\}, A\sqcup \{a+2, \dots, a+c+1\}, \dots, A\sqcup \{b-c+1, \dots, b\}\}$, then we say that $A$ is the head of $\mathcal{A}$ and $\mathcal{A}$ is $c$-sequential from $a+c$ to $b$, write $A_1\overset{c}{\prec} A_2\overset{c}{\prec}\cdots\overset{c}{\prec}A_{b-a-c+1}$, where $A_1=A\sqcup \{a+1, \dots, a+c\}, A_{b-a-c+1}=A\sqcup \{b-c+1, \dots, b\}$.
For any $A_i$,  $A_j$ contained in $\mathcal{A}$, we also say $A_i$, $ A_j$ are $c$-sequential.
 In particular, if $l_2=l_1+1$, we write $A_{l_1}\overset{c}{\prec}  A_{l_2}$; if $\max A_{l_2}=n$, write $A_{l_1}\overset{c}{\longrightarrow}  A_{l_2}$. Note that if $|\mathcal{A}|=1$, then $\mathcal{A}$ is $c$-sequential for any $c\in [k]$. Let $\mathcal{F}$ be a family and $F_1, F_2\in \mathcal{F}$. If $F_1\precneqq F_2$ and there is no $F'\in \mathcal{F}$ such that $F_1\precneqq F' \precneqq F_2$, then we say $F_1<F_2$ in $\mathcal{F}$, or $F_1<F_2$ simply if there is no confusion.

Let $R$ and $R'$  satisfy $R\prec R'$ with the corresponding
partners $T$ and $T'$ respectively. In order to measure $f_i(R')-f_i(R)$, we  define
\begin{align}\label{eq6}
&\alpha(R, R'):=d_i|\mathcal{L}([n], R', k_i)|-d_i|\mathcal{L}([n], R, k_i)|,\\ \label{eq7}
&\beta(R, R'):=\sum_{j\ne i}d_j(|\mathcal{L}([n], T, k_j)|-|\mathcal{L}([n], T', k_j)|).
\end{align}
Consequently,
\begin{align*}
f_i(R')-f_i(R)=\alpha(R, R')-\beta(R, R').
\end{align*}

\begin{proposition}\label{clm23}
Let $F<G\in \mathbb{R}_i$ and $\max G=q$. Then \\
(i) $\alpha(F, G)=d_i$.\\
(ii) $\beta(F, G)=\sum_{j\ne i}d_j{n-q\choose k_j-(q-k_i)}$.\\
 (iii) If $n= k_i+k_j$ holds for any $j\in [t]\setminus \{i\}$, then $\beta(F, G)=\sum_{j\ne i}d_j$; otherwise, we have $\beta(F, G)\geq0$ and $\beta(F, G)$ decrease as $q$ increase.
\end{proposition}

\begin{proof}
By the definition of $\alpha(F, G)$, we can get (i).  If (ii) holds,  by a direct calculation, we can get (iii). So we next to confirm (ii).
Let $F'$ and  $G'$ be the partners of $F$ and $G$ respectively. We have the following two cases.

Case (i) $\max F<n$. In this case, $\max F=q-1$ and $F\setminus \{q-1\}=G\setminus \{q\}$. By (\ref{eq7}) and Proposition \ref{prop9}, we have
\begin{align*}
\beta(F, G)
&=\sum_{j\ne i}d_j(|\mathcal{L}([n], F', k_j)|-|\mathcal{L}([n], G', k_j)|)\\
&=\sum_{j\ne i}d_j\left[{n-(q-1)\choose k_j-(q-k_i)}-{n-q\choose k_j-(q-k_i+1)}\right]\\
&=\sum_{j\ne i}d_j{n-q\choose k_j-(q-k_i)},
\end{align*}
as desired.

Case (ii) $\max F=n$. Let $p$ be the last element of $F$ not continuing to $n$. Then $G=(F\cap [p-1])\cup \{p+1, p+2, \dots, q\}$. Let $\widetilde{F}=F\cap [p]$ and $\widetilde{F'}$ be the partner of $\widetilde{F}$. It follows from Remark \ref{coro11} that
$$\sum_{j\ne i}|\mathcal{L}([n], F', k_j)|=\sum_{j\ne i}|\mathcal{L}([n], \widetilde{F'}, k_j)|.$$
Therefore,
\begin{align*}
\beta(F, G)
&=\sum_{j\ne i}d_j(|\mathcal{L}([n], F', k_j)|-|\mathcal{L}([n], G', k_j)|)\\
&=\sum_{j\ne i}d_j(|\mathcal{L}([n], \widetilde{F'}, k_j)|-|\mathcal{L}([n], G', k_j)|)\\
&=\sum_{j\ne i}d_j\left\{ {n-p\choose k_j-p+(k_i-q+p)}\right.-\left[ {n-(p+1)\choose k_j-(p+1)+(k_i-q+p)}\right.\\
& \quad \left.\left.+{n-(p+2)\choose k_j-(p+1)+(k_i-q+p)}+\cdots+{n-q\choose k_j-(p+1)+(k_i-q+p)}\right]  \right\}\\
&=\sum_{j\ne i}d_j\left\{{n-p\choose k_j-(q-k_i)}-\left[ {n-p\choose k_j-(q-k_i)}-{n-q\choose k_j-(q-k_i)}\right]\right\}\\
&=\sum_{j\ne i}d_j{n-q\choose k_j-(q-k_i)},
\end{align*}
as desired.
\end{proof}

By Proposition \ref{clm23}, we can see the following two claims.

\begin{remark}\label{remark2.12}
Suppose that $n=k_i+k_j$ holds for all $j\in [t]\setminus \{i\}$ and $\sum_{j\ne i}d_j>d_i$. Then $\sum_{j=1}^td_j|\mathcal{A}_j|\leq \sum_{j=1}^td_j {n-1\choose k_j-1}$.
The equality holds if and only if $\mathcal{A}_j=\{A: 1\in A, |A|=k_j\}$ and for all $j\in [t]$.
\end{remark}

\begin{remark}\label{remark2.13}
Suppose that $n=k_i+k_j$ holds for all $j\in [t]\setminus \{i\}$ and $\sum_{j\ne i}d_j=d_i$. Then $\sum_{j=1}^td_j|\mathcal{A}_j|\leq \sum_{j=1}^td_j {n-1\choose k_j-1}$.
Let $F\in \mathbb{R}_i$ and $T$ be the partner of $F$.
The equality holds for $\mathcal{A}_i=\mathcal{L}([n], F, k_i)$ and $\mathcal{A}_j=\mathcal{L}([n], T, k_j)$ for each $j\in [t]\setminus \{i\}$.
\end{remark}

From the above two claims, we may assume that if $n=k_i+k_j$ holds for all $j\in [t]\setminus \{i\}$, then  $\sum_{j\ne i}d_j<d_i$.

We will prove the following four crucial lemmas showing some `local modularity' of  -$f_i(R)$ in Section \ref{sec3}.

\begin{lemma}\label{clm3}
Let $c\in [k_i]$ and $F, G, H\in \mathbb{R}_i$ with $F\overset{c}{\prec} G\overset{c}{\prec} H$. Assume that  if $n=k_i+k_j$ holds for all $j\in [t]\setminus \{i\}$, then  $\sum_{j\ne i}d_j<d_i$. If $\alpha(F, G)\geq\beta(F, G)$, then $\alpha(G, H)>\beta(G, H)$. This means that $f_i(G)\geq f_i(F)$ implies $f_i(H)>f_i(G)$.
\end{lemma}

Denote $\mathcal{R}_{i, k}=:\{R\in\mathbb{R}_i: [n-k+1, n]\subset R\}$, and
$\mathcal{R}_i(k)=:\{R\setminus [n-k+1, n]: R\in\mathcal{R}_{i, k}\}$ for $k\in [k_i-1]$. In addition, we will write $\mathcal{R}_i(0)=\mathbb{R}_i$. When we consider $f_i(R), \alpha(R, T)$ and $\beta(R, T)$ for $R, T\in \mathcal{R}_{i, k}$, we simply write $f_i(R\setminus[n-k+1, n])$ etc. In particular,   $f_i(\{1\})$ is indeed $f_i(\{1, n-k_i+1, n-k_i+2, \dots, n\})$, and $f_i(\{m_i\})$ is indeed $f_i(\{m_i, n-k_i+1, n-k_i+2, \dots, n\})$.

\begin{lemma}\label{clm28}
For any $j\in [0, k_i-1]$, let $1\leq c\leq k_i-j$ and $F, G, H\in \mathcal{R}_i(j)$ with $F\overset{c}{\prec}G\overset{c}{\prec}H$. Assume that  if $n=k_i+k_j$ holds for all $j\in [t]\setminus \{i\}$, then  $\sum_{j\ne i}d_j<d_i$.
  If $\alpha(F, G)\geq\beta(F, G)$, then $\alpha(G, H)>\beta(G, H)$. This means that $f_i(G)\geq f_i(F)$ implies $f_i(H)>f_i(G)$.
\end{lemma}

\begin{lemma}\label{clm29}
Suppose $k_i\geq 2$. Let $3\leq j\leq k_i+1$. Assume that  if $n=k_i+k_j$ holds for all $j\in [t]\setminus \{i\}$, then  $\sum_{j\ne i}d_j<d_i$. If $f_i(\{2, 3, \dots, j\})f_i(\{2, 3, \dots, j-1\})\geq f_i(\{2, 3, \dots, j\})$, then $f_i(\{2, 3, \dots, j-2\})>f_i(\{2, 3, \dots, j-1\})$.
\end{lemma}

\begin{lemma}\label{clm30}
Let $m_i+1\leq j\leq m_i+k_i-1$. Assume that  if $n=k_i+k_j$ holds for all $j\in [t]\setminus \{i\}$, then  $\sum_{j\ne i}d_j<d_i$. If $f_i(\{m_i, m_i+1, \dots, j-1\})\geq f_i(\{m_i, m_i+1, \dots, j\})$, then $f_i(\{m_i, m_i+1, \dots, j-2\})>f_i(\{m_i, m_i+1, \dots, j-1\})$.
\end{lemma}


Combining these four lemmas, we will be able to prove Theorem \ref{2}. Let us be precise below.

For a family $\mathcal{F}$, denote $f(\mathcal{F})=\max \{f(F): F\in \mathcal{F}\}$. Applying  Lemma \ref{clm3} repeatedly, we have
\begin{equation}\label{frr}
f_i(\mathbb{R}_i)=\max \{f_i(\{2, 3, \dots, k_i+1\}), f_i(\{m_i, m_i+1, \dots, m_i+k_i-1\}), f_i(\mathcal{R}_i(1))\}.
\end{equation}
(Let us explain the above observation. For example, suppose that $k_i=3, a< b <c \in [n]$ and $\{a, b, c\}\in \mathbb{R}_i$. Applying Lemma \ref{clm3}, we have
\begin{align*}
f_i(\{a, b, c\})&\leq \max \{f_i(\{a, b, b+1\}, f(_i\{a, b, n\})\}\\
&= \max \{f_i(\{a, b, b+1\}, f_i(\mathcal{R}_i(1))\}\\
&= \max\{ f_i(\{a, a+1, a+2\}), f_i(\mathcal{R}_i(1))\}\\
&= \max \{ f_i(\{2, 3, 4\}), f_i(\{m_i, m_i+1, m_i+2\}), f_i(\mathcal{R}_i(1))\}.)
\end{align*}
Similarly, applying Lemma \ref{clm28} repeatedly, we have
\begin{align}\nonumber
&f_i(\mathcal{R}_i(1))=\max \{f_i(\{2, 3, \dots, k_i\}), f_i(\{m_i, m_i+1, \dots, m_i+k_i-2\}), f_i(\mathcal{R}_i(2))\},\\ \nonumber
&f_i(\mathcal{R}_i(2))=\max \{f_i(\{2, 3, \dots, k_i-1\}), f_i(\{m_i, m_i+1, \dots, m_i+k_i-3\}), f_i(\mathcal{R}_i(3))\},\\\nonumber
&\quad \vdots\\ \label{fr}
&f_i(\mathcal{R}_i(k_i-1))=\max \{f_i(\{1\}),  f_i(\{m_i\})\}.
\end{align}
By Lemma \ref{clm29}, we have
\begin{align}\label{dd}
&\max\{f_i(\{2, 3, \dots, k_i+1\}), f_i(\{2, 3, \dots, k_i\}), \dots, f_i(\{2, 3\}), f_i(\{2\})\} \nonumber \\
&=\max\{f_i(\{2, 3, \dots, k_i+1\}), f_i(\{2\})\}\nonumber \\
&\leq\max\{f_i(\{2, 3, \dots, k_i+1\}), \max\{f_i(\{1\}), f_i(\{m_i\})\}\}
\end{align}
By Lemma \ref{clm30} , we have
\begin{align}\label{d}
&\max\{f_i(\{m_i, m_i+1, \dots, m_i+k_i-1\}), f_i(\{m_i, m_i+1, \dots, m_i+k_i-2\}), \dots,   f_i(\{m_i\})\} \nonumber \\
&=\max\{f_i(\{m_i, m_i+1, \dots, m_i+k_i-1\}),         f_i(\{m_i\})      \}.
\end{align}

\begin{clm}\label{claim2.14new}
Let $k\in [2, m_i]$. If $f_i([k, k+k_i-1])> f_i(\{k-1\}\cup [n-k_i+2, n])$, then $f_i(R)$ is increase on $[k, k+k_i-1]\prec R \prec [k, k+k_i-2]\cup\{n\}$.
\end{clm}

\begin{proof}
Since $f_i([k, k+k_i-1])> f_i(\{k-1\}\cup [n-k_i+2, n])$,
$$\beta(\{k-1\}\cup [n-k_i+2, n], [k, k+k_i-1])< \alpha(\{k-1\}\cup [n-k_i+2, n], [k, k+k_i-1]).$$
Note that $\{k-1\}\cup [n-k_i+2, n]<[k, k+k_i-1]$ in $\mathbb{R}_i$ and $\max R$ increase on $[k, k+k_i-1]\prec R \prec [k, k+k_i-2]\cup\{n\}$.
By Lemma \ref{clm23}, we can see that $f_i(R)$ is increase on $[k, k+k_i-1]\prec R \prec [k, k+k_i-2]\cup\{n\}$.
\end{proof}

If
$$\beta(\{m_i-1\}\cup [n-k_i+2, n], [m_i, m_i+k_i-1])\geq \alpha(\{m_i-1\}\cup [n-k_i+2, n], [m_i, m_i+k_i-1]),$$
then
\begin{equation}\label{eq35}
f_i([m_i, m_i+k_i-1])\leq f_i(\{m_i-1\}\cup [n-k_i+2, n])=f_i(\{m_i-1\})\leq\max \{f_i(\{1\}),  f_i(\{m_i\})\}.
\end{equation}
Otherwise,
$$\beta(\{m_i-1\}\cup [n-k_i+2, n], [m_i, m_i+k_i-1])< \alpha(\{m_i-1\}\cup [n-k_i+2, n], [m_i, m_i+k_i-1]).$$
So $f_i([m_i, m_i+k_i-1])> f_i(\{m_i-1\}\cup [n-k_i+2, n])$.
By Claim \ref{claim2.14new} and Lemma \ref{clm30}
\begin{equation}\label{eq35+}
f_i([m_i, m_i+k_i-1])<f_i([m_i, m_i+k_i-2]\cup\{n\})=f_i([m_i, m_i+k_i-2])<f_i(\{m_i\}).
\end{equation}
Combining (\ref{frr})--(\ref{eq35+}), we have
\begin{equation}\label{eq33}
f_i(\mathbb{R}_i)=\max \{f_i(\{2, 3, \dots, k_i+1\}), f_i(\{1\}),  f_i(\{m_i\})\}.
\end{equation}
If
$\beta((\{1\}, \{2, 3,\dots, k_i+1\}))\geq \alpha(\{1\}, \{2, 3,\dots, k_i+1\}),$
then $f_i(\{1\})\geq f_i(\{2, 3,\dots, k_i+1\})$. Otherwise,
$\beta((\{1\}, \{2, 3,\dots, k_i+1\}))< \alpha(\{1\}, \{2, 3,\dots, k_i+1\}).$
So $f_i([2, k_i+1])> f_i(\{1\})$. By Claim \ref{claim2.14new} and Lemma \ref{clm30}
again, $f_i([2, k_i+1])<f_i(\{2\})$.
Combining with  (\ref{eq33}), we have
\begin{equation}\label{ccc}
f_i(\mathbb{R}_i)=\max \{ f_i(\{1\}),  f_i(\{m_i\})\}.
\end{equation}
This complete the quantitative part of Theorem \ref{2}. What left is to discuss when the equality holds in the above equality. We will meet the following three cases.


Case (i) $d_i{n\choose k_i}-d_i{n-m_i\choose k_i}+\sum_{j\ne i}d_j{n-m_i\choose k_j-m_i}>\sum_{j=1}^td_j{n-1\choose k_j-1}$.\\
In this case $\sum_{j=1}^td_j|\mathcal{A}_j|=d_i{n\choose k_i}-d_i{n-m_i\choose k_i}+\sum_{j\ne i}d_j{n-m_i\choose k_j-m_i}.$
Combining Lemma \ref{clm28} and Lemma \ref{clm30}, we have $\sum_{j=1}^td_j|\mathcal{A}_j|=f_i(\{m_i\})$. In view of (\ref{fm}), we have $|\mathcal{A}_i|={n\choose k_i}-{n-m_i\choose k_i}$ and $|\mathcal{A}_j|={n-m_i\choose k_j-m_i}$ for $j\in[2, t]\setminus\{i\}$, in particular, $|\mathcal{A}_t|=1$. Let $\mathcal{A}_t=\{T\}$ for some $T\in {[n]\choose k_t}$.
Since $\mathcal{A}_i$ and $\mathcal{A}_t$ are cross-intersecting and $|\mathcal{A}_i|={n\choose k_i}-{n-k_t\choose k_i}$, we have $\mathcal{A}_i=\{F\in {[n]\choose k_i}: F\cap T\ne \emptyset\}$. Since $\mathcal{A}_j$ and $\mathcal{A}_t$ are cross-intersecting and  $|\mathcal{A}_j|={n-k_t\choose k_j-k_t}$ for $j\neq i$, we get $\mathcal{A}_j=\{F\in {[n]\choose k_j}: T\subset F\}$ for $j\neq i$. As desired.

Case (ii) $d_i{n\choose k_i}-d_i{n-m_i\choose k_i}+\sum_{j\ne i}d_j{n-m_i\choose k_j-m_i}<\sum_{j=1}^td_j{n-1\choose k_j-1}$.\\
In this case
$\sum_{j=1}^td_j|\mathcal{A}_j|=\sum_{j=1}^td_j{n-1\choose k_j-1}.$
To deal with this case, we will use a result proved independently  by F\"uredi and Griggs \cite{FG} and
M\"ors\cite{M}. To state it, we need a definition. For two integers $i$ and $j$ with $n\geq i+j$ and a family $\mathcal{F}\subseteq {[n]\choose i}$.
Denote
$$ \mathcal{D}_j(\mathcal{F})=:\left\{ D\in {[n]\choose j}: \exists F\in \mathcal{F} \ {\rm such \ that} \ D\cap F=\emptyset   \right\}.$$

\begin{proposition}[F\" uredi, Griggs \cite{FG}, M\"ors\cite{M}]\label{FGM}
Suppose that $n>k+l$, $\mathcal{B}\subseteq {[n]\choose l}$ with $|\mathcal{B}|={n-r\choose l-r}$ for some $1\leq r\leq l$. Then
$$\mathcal{D}_k(\mathcal{B})\geq {n-r\choose k}$$
with strictly inequality unless for some $R\in {[n]\choose r}$, $\mathcal{B}=\{B\in {[n]\choose l}: R\subseteq B \}$.
\end{proposition}

Combining Lemma \ref{clm28} and Lemma \ref{clm30}, we have $\sum_{j=1}^td_j|\mathcal{A}_j|=f_i(\{1\})$ and $|\mathcal{A}_j|={n-1\choose k_j-1}$ for each $j\in [t]$.
Since $\mathcal{A}_i$ and $\mathcal{A}_j$ are cross-intersecting for all $j\in [t]\setminus \{i\}$,  $\mathcal{A}_i\cap \mathcal{D}_{k_i}(\mathcal{A}_j)=\emptyset$. Therefore, $\vert  \mathcal{D}_i(\mathcal{A}_j) \vert \le {n \choose k_i}-{n-1 \choose k_i-1}={n-1 \choose k_i}$. If there is some $j\in [t]\setminus\{i\}$ such that $n>k_i+k_j$, then by taking $r=1$ in Proposition \ref{FGM}, there exists some $a\in [n]$ such that $\mathcal{A}_j=\{A\in {[n]\choose k_j}: a\in A\}$. So $\mathcal{A}_i=\{A\in {[n]\choose k_i}: a\in A\}$. Since $n\geq k_i+k_j$ for each $j\in[t]\setminus \{i\}$, $\mathcal{A}_j=\{A\in {[n]\choose k_j}: a\in A\}$.
If $n=k_i+k_j$ holds for every $j\in [t]\setminus \{i\}$ and $\sum_{j\ne i}d_j> d_i$, then by Remark \ref{remark2.12}, $\mathcal{A}_j=\{A: a\in A, |A|=k_j\}$ and for all $j\in [t]$ and for some fixed $a$. If $n=k_i+k_j$ holds for every $j\in [t]\setminus \{i\}$ and $\sum_{j\ne i}d_j=d_i$, then we will meet  Case (iii). If $n=k_i+k_j$ holds for every $j\in [t]\setminus \{i\}$ and $\sum_{j\ne i}d_j< d_i$, then we will meet Case (i).
As desired.

Case (iii) $d_i{n\choose k_i}-d_i{n-m_i\choose k_i}+\sum_{j\ne i}d_j{n-m_i\choose k_j-m_i}=\sum_{j=1}^td_j{n-1\choose k_j-1}$.\\
In this case
we first consider the case that $t=2$ and $n=k_i+k_{3-i}$. Suppose that $\mathcal{A}_i$ and $\mathcal{A}_{3-i}$ are the families such that $d_i|\mathcal{A}_i|+d_{3-i}|\mathcal{A}_{3-i}|$ is the maximum.  Note that a set $A\in {k_1+k_2 \choose k_1}$ intersects with any set $B\in {k_1+k_2 \choose k_2}$ except $B=\overline{A}$, that is, $|\mathcal{A}_1|+|\mathcal{A}_2|={n\choose k_1}={n-1\choose k_1-1}+{n-1\choose k_2-1}$. If $d_i\leq d_{3-i}$,  then $\mathcal{A}_{3-i}\subseteq {[n]\choose k_{3-i}}$ with $|\mathcal{A}_{3-i}|={n-1\choose k_{3-i}-1}$ and $\mathcal{A}_i={[n]\choose k_i}\setminus \overline{\mathcal{A}_{3-i}}$. If $d_i>d_{3-i}$,  then
there is some $k_{3-i}$-set $B\subseteq [n]$, such that $\mathcal{A}_{3-i}=\{B\}$ and $\mathcal{A}_i=\{A\in {[n]\choose k_i}: A\cap B\ne \emptyset\}$. This implies that we meet (1).
As required.

We next assume that if $t=2$, then $n>k_i+k_{3-i}$. If $n=k_i+k_j$ holds for every $j\in [t]\setminus \{i\}$ and $\sum_{j\ne i}d_j=d_i$, then by Remark \ref{remark2.13},  $\mathcal{A}_j=\mathcal{A}$ for all $j\in [t]\setminus \{i\}$, where $\mathcal{A}\subseteq {[n]\choose k}$ is an intersecting family with size $|\mathcal{A}|={n-1\choose k-1}$, and $\mathcal{A}_i={[n]\choose k_i}\setminus \overline{\mathcal{A}}$.
Otherwise, by Lemma \ref{clm28} and Lemma \ref{clm30},  either there exits some $a\in [n]$ such that $\mathcal{A}_j=\{A\in {[n]\choose k_j}: a\in A\}$ holds for every $j\in [t]$ or
there exists some $k_t$-set $T$ such that $\mathcal{A}_i=\{F\in {[n]\choose k_i}: F\cap T\ne \emptyset\}$ and $\mathcal{A}_j=\{F\in {[n]\choose k_j}: T\subset F\}$ for each $j\in [t]\setminus\{i\}$. As desired.


\subsection{Proofs of Theorems \ref{2+} and \ref{2++} }
\begin{proof}[Proof of Theorem \ref{2+}]
If each $\mathcal{A}_i$, $i\in [t]$ has size $|\mathcal{A}_i|\leq {n-1\choose k_i-1}$, then
$$\sum_{i=1}^td_i|\mathcal{A}_i|\leq \sum_{i=1}^td_i{n-1\choose k_i-1},$$
as desired.
Otherwise, we meet the condition of Theorem \ref{2}. Then Theorem \ref{2+} can be derived from Theorem \ref{2} immediately.
\end{proof}

\begin{proof}[Proof of Theorem \ref{2++}]
Recall that in Theorem \ref{2++}, $k_1\geq\dots\geq k_t$, $d_1\geq\dots\geq d_t$ and $n\geq k_1+k_2$.
To see this theorem, we only need to confirm the following proposition and we will prove it at the end of this paper.

\begin{proposition}\label{g}
Suppose that $n\geq k_1+k_2 $, $k_1\geq k_2\geq \dots \geq k_t$ are positive integers and  $d_1\geq\dots\geq d_t$ are positive constants. Let $i\in [t]$, then for $1\leq s \leq k_t$, we have
\[
f_1(\{s\})=\max \{f_j(\{s\}): j\in [t]\}.
\]
In particular,
\begin{align*}
&f_1(\{1\})=\max \{f_j(\{1\}): j\in [t]\},\\
&f_1(\{m_1\})=\max \{f_j(\{m_j\}): j\in [t]\}.
\end{align*}
\end{proposition}

We owe the proofs of Proposition \ref{g},  and Lemmas \ref{clm3}, \ref{clm28},  \ref{clm29} and \ref{clm30}. The proof of Proposition \ref{g} will be given in Section \ref{subprooflemma}, and the proofs of Lemmas \ref{clm3}, \ref{clm28},  \ref{clm29} and \ref{clm30}  will be given in Section \ref{sec3}.

\subsection{Proof of Proposition \ref{g}}\label{subprooflemma}

\noindent
{\em Proof of Proposition \ref{g}.}
Note that for each $j\in [t]$, we have
\begin{equation}\label{cc}
f_j(\{1\})=\sum_{q=1}^td_j{n-1\choose k_q-1},
\end{equation}
since $\mathcal{A}_j$ is the family of all sets  having lex order smaller than or equal to $\{1\}$, this means that $\mathcal{A}_j$ is the full star containing $1$. Consequently, all sets in other $\mathcal{A}_l$ are also the full star containing $1$ since they are pairwise cross-intersecting.
So $f_1(\{1\})=\max \{f_j(\{1\}): j\in [t]\}$.

We first deal with the case $t=2$. In this case, $m_1=k_2$, $m_2=k_1$ and
\begin{align*}
&f_1(\{m_1\})=d_1{n-1\choose k_1-1}+\dots+d_1{n-k_2\choose k_1-1}+d_2,\\
&f_2(\{m_2\})=d_2{n-1\choose k_2-1}+\dots+d_2{n-k_1\choose k_2-1}+d_1.
\end{align*}
For $q\geq k_1+k_2$, denote
$$h(q)={n-1\choose k_1-1}+\dots+{n-k_2\choose k_1-1}-{n-1\choose k_2-1}-\dots-{n-k_1\choose k_2-1}.$$
Since $q\geq k_1+k_2$ and $k_1\geq k_2$, $h(q)$ increases as $q$ increases. Notice that when $q=k_1+k_2$, $h(q)=0$, so $h(q)\geq 0$.
Thus
$$f_1(\{m_1\})-f_2(\{m_2\})\geq (d_1-d_2)\bigg({n-1\choose k_2-1}+\dots+{n-k_1\choose k_2-1}-1\bigg)\geq 0.$$
As desired.

Next we may assume that $t>2$. So $m_1=m_2=k_t$.
We next prove that for $2\leq s\leq k_t$,
$f_1(\{s\})=\max \{f_j(\{s\}): j\in [t]\}.$

Since $n\geq k_1+k_2$ and $k_1\geq k_2\geq \cdots \geq k_t$, we  only need to prove that
\begin{equation}\label{f}
f_1(\{s\})\geq f_2(\{s\}).
\end{equation}
By the definition of $f_j(R)$, we have
\begin{align*}
f_1(\{s\})=d_1{n-1\choose k_1-1}+\cdots+d_1{n-s\choose k_1-1}+\sum_{j=2}^td_j{n-s\choose k_j-s},
\end{align*}
and
\begin{align*}
f_2(\{s\})=d_2{n-1\choose k_2-1}+\cdots+d_2{n-s\choose k_2-1}+\sum_{j\ne 2, j=1}^td_j{n-s\choose k_j-s}.
\end{align*}

We denote
$$g(n)={n-1\choose k_1-1}+\cdots+{n-s\choose k_1-1}-{n-1\choose k_2-1}-\cdots-{n-s\choose k_2-1}+{n-s\choose k_2-s}-{n-s\choose k_1-s}.$$

\begin{clm}\label{pp}
For any integer $q$ with $q\geq k_1+k_2$ and $k_1\geq k_2$, we have
\begin{equation}\label{p}
g(q+1)-g(q)\geq 0.
\end{equation}
\end{clm}
{\em Proof of Claim \ref{pp}.}
Indeed,
\begin{align*}
g(q)&={q-1\choose k_1-1}+\cdots+{q-s\choose k_1-1}+{q-s\choose k_2-s}\\
&\quad -\left\{{q-1\choose k_2-1}+\cdots+{q-s\choose k_2-1}+{q-s\choose k_1-s}\right\}\\
&={q-2\choose k_1-1}+\cdots+{q-s\choose k_1-1}+\sum_{j=2}^{s}{q-j\choose k_1-j+1}\\
&\quad -\left\{{q-2\choose k_2-1}+\cdots+{q-s\choose k_2-1}+\sum_{j=2}^{s}{q-j\choose k_2-j+1}\right\},
\end{align*}
and
\begin{align*}
g(q+1)&={q-1\choose k_1-1}+\cdots+{q+1-s\choose k_1-1}+\sum_{j=2}^{s}{q+1-j\choose k_1-j+1}\\
&\quad -\left\{{q-1\choose k_2-1}+\cdots+{q+1-s\choose k_2-1}+\sum_{j=2}^{s}{q+1-j\choose k_2-j+1}\right\}.
\end{align*}
Since $q\geq k_1+k_2$ and $k_1\geq k_2$, then for all $j\geq 0$, we have
\begin{equation}\label{fff}
{q-j\choose k_1-j}\geq {q-j\choose k_2-j}.
\end{equation}
This gives
\begin{align*}
&\sum_{j=2}^{s}{q+1-j\choose k_1-j+1}-\sum_{j=2}^{s}{q+1-j\choose k_2-j+1}
-\sum_{j=2}^{s}{q-j\choose k_1-j+1}+\sum_{j=2}^{s}{q-j\choose k_2-j+1}\\
&=\sum_{j=2}^{s}\left\{{q-j\choose k_1-j}-{q-j\choose k_2-j}\right\}\\
&\geq 0.
\end{align*}
Hence, to get (\ref{p}), it is sufficient to show the following claim.
\begin{clm}\label{?}
For any integer $q$ with $q\geq k_1+k_2$ and $k_1\geq k_2$, we have
\begin{equation}\label{ffff}
{q-1\choose k_1-1}-{q-s\choose k_1-1}\geq {q-1\choose k_2-1}-{q-s\choose k_2-1}.
\end{equation}
\end{clm}
\begin{proof}[Proof of Claim \ref{?}]
If ${q-s\choose k_1-1}\leq {q-s\choose k_2-1}$, then applying (\ref{fff}) for $j=1$, we have ${q-1\choose k_1-1}\geq {q-1\choose k_2-1}$, so the desired inequality (\ref{ffff}) holds. Suppose that ${q-s\choose k_1-1}> {q-s\choose k_2-1}$. Since $k_1\geq k_2$ and $q\geq k_1+k_2$, we have
$$\frac{{q-s\choose k_1-2}}{{q-s\choose k_1-1}}\geq \frac{{q-s\choose k_2-2}}{{q-s\choose k_2-1}}, $$
this implies that ${q-s\choose k_1-2}> {q-s\choose k_2-2}$. Similarly, for all $j\geq 0$,
$$\frac{{q-s+j\choose k_1-2}}{{q-s\choose k_1-2}}\geq \frac{{q-s+j\choose k_2-2}}{{q-s\choose k_2-2}}.$$
So ${q-s+j\choose k_1-2}\geq {q-s+j\choose k_2-2}$ holds for any $j\geq 0$, yielding
\begin{align*}
&{q-1\choose k_1-1}-{q-s\choose k_1-1}-\left\{  {q-1\choose k_2-1}-{q-s\choose k_2-1}\right\}\\
&=\sum_{j=2}^{s}\left\{ {q-j\choose k_1-2}-{q-j\choose k_2-2} \right\}\\
&\geq 0,
\end{align*}
as desired.
\end{proof}
This complete the proof of Claim \ref{pp}.
\end{proof}
Note that $g(k_1+k_2)=0$. Then Claim \ref{pp} gives
$${n-1\choose k_1-1}+\cdots+{n-s\choose k_1-1}\geq{n-1\choose k_2-1}+\cdots+{n-s\choose k_2-1}+{n-s\choose k_1-s}-{n-s\choose k_2-s}.$$
Thus,
$$f_1(\{s\})-f_2(\{s\})\geq (d_1-d_2)\Big({n-1\choose k_2-1}+\cdots+{n-s\choose k_2-1}-{n-s\choose k_2-s}\Big)\geq0.$$
This proves (\ref{f}) and completes the proof of Proposition \ref{g}.

\medskip

\section{Verifying unimodality: proofs of Lemmas \ref{clm3}, \ref{clm28},  \ref{clm29}, \ref{clm30}}\label{sec3}
We show some preliminary properties. We need the following preparation.

\begin{clm}\label{clm20}
Let $F_1, F_2, F'_1, F'_2\in \mathbb{R}_i, c\in [k_i], F_1\overset{c}{\prec} F_2$ and $F'_1\overset{c}{\prec} F'_2$. If $\max F_1=\max F'_1$, then $\alpha(F_1, F_2)=\alpha(F'_1, F'_2)$ and $\beta(F_1, F_2)=\beta(F'_1, F'_2)$.
\end{clm}

\begin{proof}
Let $A$ be the head of $F_1$ and $F_2$, $A'$ be the head of $F'_1$ and $F'_2$ and let $\max F_1=\max F_2=q$. Then $\max F_2=\max F'_2=q+1.$ It is easy to see that $F_1\setminus A=F'_1\setminus A'$ and $F_2\setminus A=F'_2\setminus A'$, by  Proposition \ref{prop9}, we conclude that $\beta(F_1, F_2)=\beta(F'_1, F'_2)$. Let $G_1, G_2, G'_1, G'_2$ be the partners of $F_1, F_2, F'_1, F'_2$ respectively. Then $G_1\setminus G_2=G'_1\setminus G'_2$ and $G_2\setminus G_1=G'_2\setminus G'_1$, by Proposition \ref{prop9}, we have $\alpha(F_1, F_2)=\alpha(F'_1, F'_2)$, as promised.
\end{proof}

\begin{clm}\label{clm21}
Let $F, H, G\in \mathbb{R}_i$ with $F\prec H\prec G$. Then $\alpha(F, G)=\alpha(F, H)+\alpha(H, G)$ and $\beta(F, G)=\beta(F, H)+\beta(H, G)$.
\end{clm}

\begin{proof}
By (\ref{eq6}), we have
\begin{align*}
\alpha(F, H)+\alpha(H, G)
&=d_i\big(|\mathcal{L}([n], H, k_i)|-|\mathcal{L}([n], F, k_i)|+|\mathcal{L}([n], G, k_i)|-|\mathcal{L}([n], H, k_i)|\big)\\
&=d_i\big(|\mathcal{L}([n], G, k_i)|-|\mathcal{L}([n], F, k_i)|\big)\\
&=\alpha(F, G),
\end{align*}
as desired.
Let $F', H', G'$ be the partners of $F, H, G$ respectively. Then by (\ref{eq7}), we have
\begin{align*}
\beta(F, H)+\beta(H, G)
&=\sum_{j\ne i}d_j\big(|\mathcal{L}([n], F', k_j)|-|\mathcal{L}([n], H', k_j)|+|\mathcal{L}([n], H', k_j)|-|\mathcal{L}([n], G', k_j)|\big)\\
&=\sum_{j\ne i}d_j\big(|\mathcal{L}([n], F', k_j)|-|\mathcal{L}([n], G', k_j)|\big)\\
&=\beta(F, G),
\end{align*}
as desired.
\end{proof}

By Claims \ref{clm20} and \ref{clm21}, the following corollary is obvious.

\begin{corollary}\label{coro22}
Let $c\in [k_i]$ and $F, G, F', G'\in \mathbb{R}_i$. If $F, G$ are $c$-sequential, $F', G'$ are $c$-sequential and $\max F=\max F', \max G=\max G'$, then $\alpha (F, G)=\alpha (F', G')$ and $\beta (F, G)=\beta (F', G')$.
\end{corollary}

\begin{clm}\label{clm24}
Let $2\leq c\leq k_i$ and $F, G, H, F_1\in \mathbb{R}_i$ with $F\overset{c}{\prec} G\overset{c}{\prec} H, F\overset{c-1}{\prec}F_1$ and $\max F =q$. Then
\begin{align*}
&\alpha(F, G)=\alpha(F, F_1)+\alpha(G, H),\\
&\beta(F, G)=\beta(F, F_1)+\beta(G, H)+\sum_{j\ne i}d_j{n-(q+2)\choose k_j-(q-k_i+1)}.
\end{align*}
\end{clm}

\begin{proof}
Let $A$ be the head of $F, G, H$. Then by the definition, $F=A\sqcup \{q-c+1,\dots, q\}, G=A\sqcup \{q-c+2,\dots, q+1\}, H=A\sqcup \{q-c+3,\dots, q+2\}$ and $F_1=A\sqcup \{q-c+1\}\sqcup \{q-c+3,\dots, q+1\}$. Define $F_2$ as $F_2<G$ in $\mathbb{R}_i$. Since $G\setminus A$ continues and $q-c+1\not\in G$, then $q-c+1\in F_2$. Hence, $F_2=A\cup \{q-c+1, n-c+2, n-c+3, \dots, n\}$. Similarly, define $F_3$ as $F_3<H$ in $\mathbb{R}_i$. Then $F_3=A\sqcup \{q-c+2, n-c+2, n-c+3, \dots, n\}$. Moreover, $F_1$ and $F_2$ are $(c-1)$-sequential, and $G$ and $F_3$ are $(c-1)$-sequential. Clearly, $\max F_1=\max G=q+1$ and $\max F_2=\max F_3=n$. By Corollary \ref{coro22}, we have $\alpha(F_1, F_2)=\alpha(G, F_3)$ and $\beta(F_1, F_2)=\beta(G, F_3)$. By the definition, we get $\alpha(F_2, G)=\alpha(F_3, H)=1$. Combining with Claim \ref{clm21}, we have
\begin{align*}
\alpha(F, G)
&=\alpha(F, F_1)+\alpha(F_1, G)\\
&=\alpha(F, F_1)+\alpha(F_1, F_2)+\alpha(F_2, G)\\
&=\alpha(F, F_1)+\alpha(G, F_3)+\alpha(F_3, H)\\
&=\alpha(F, F_1)+\alpha(G, H),
\end{align*}
and
\begin{align*}
\beta(F, G)
&=\beta(F, F_1)+\beta(F_1, G)\\
&=\beta(F, F_1)+\beta(F_1, F_2)+\beta(F_2, G)\\
&=\beta(F, F_1)+\beta(G, F_3)+\beta(F_3, H)+\beta(F_2, G)-\beta(F_3, H)\\
&=\beta(F, F_1)+\beta(G, H)+\sum_{j\ne i}d_j{n-(q+2)\choose k_j-(q-k_i+1)},
\end{align*}
where the last equality follows from Proposition \ref{clm23}. More specifically, we can see that
\begin{align*}
&\beta(F_2, G)=\sum_{j\ne i}d_j{n-(q+1)\choose k_j-(q+1-k_i)}, \\
&\beta(F_3, H)=\sum_{j\ne i}d_j{n-(q+2)\choose k_j-(q+2-k_i)}.
\end{align*}
\end{proof}

\begin{definition}
Let $M\geq 2$ and $\mathcal{G}=\{G_1, G_2, \dots, G_M\}\subset \mathbb{R}_i$ with $G_1\prec G_2 \prec \dots \prec G_M$. If there is $g\in [0, M-1]$ satisfying the following two conditions:\\
(i) $f(G_{j+1})< f(G_j)$ for $1\leq j\leq g$,\\
(ii)  $f(G_{j+1})\geq f(G_j)$ for $g+1\leq j\leq M-1$,\\
then we say that $\mathcal{G}$ is a down-up family and $g$ is the down degree of $\mathcal{G}$, write $d_{\mathcal{G}}^{\downarrow}$.
\end{definition}
Recall that $i\in [t]$ is the fixed index satisfying $|\mathcal{A}_i|\geq {n-1\choose k_i-1}$.
Let $$l=\max_{j\ne i} k_j.$$

\subsection{Proof of Lemma \ref{clm3}}

To show Lemma \ref{clm3}, we need the following preparations. All arguments below are under the assumption of Lemma \ref{clm3}, i.e., assume that $c\in [k_i]$ and $F, G, H\in \mathbb{R}_i$ with $F\overset{c}{\prec} G\overset{c}{\prec} H$ satisfying $\alpha(F, G)\geq\beta(F, G)$.  We need to show that $\alpha(G, H)>\beta(G, H)$.

\begin{clm}\label{clm26.1}
Let $c'\in[k_i]$ and $R, R', T\in \mathbb{R}_i$ with $R\precneqq T\precneqq R'$. If $R, R'$ are $c'$-sequential, then $\max T\geq \max R+1$.
\end{clm}

\begin{proof}
Let $A$ be the head of $R$ and $R'$. Since $R\precneqq T\precneqq R'$, we have $A\subset T$. Since $\min R\setminus T< \min T\setminus R$ and $R\setminus A$ continues to $\max R$, we have $\max T>\max R$.
\end{proof}

Let $A$ be the head of $F, G, H$ and $\max F=q$. Then $F=A\sqcup \{q-c+1, \dots, q\}, G=A\sqcup \{q-c+2, \dots, q+1\}$ and $ H=A\sqcup \{q-c+3, \dots, q+2\}$.

Let $x\in [n]$ be the smallest integer such that $\sum_{j\ne i}d_j{n-x\choose k_j-(x-k_i)}\leq d_i$.
Since our assumption  that  if $n=k_i+k_j$ holds for all $j\in [t]\setminus \{i\}$, then  $\sum_{j\ne i}d_j<d_i$, $x$ is well-defined.
 By (ii) and (iii) of  Proposition \ref{clm23}, we can see that for all $y$ with $x< y\leq n$, we have $\sum_{j\ne i}d_j{n-y\choose k_j-(y-k_i)}<\sum_{j\ne i}d_j{n-(y-1)\choose k_j-(y-1-k_i)}\leq d_i$.
If $x\leq k_i+2$, then since $\max H\geq k_i+3$, by Proposition \ref{clm23} and the definition of $x$, we get $\alpha(G, H)>\beta(G, H)$, as desired. Thus, we next assume that $x\geq k_i+3$. On the other hand,
since $\sum_{j\ne i}d_j{n-(k_i+l+1)\choose k_j-l-1}=0< d_i$, then
$x\leq \min \{k_i+l+1, n \}$.
We have the following claim.

\begin{clm}\label{clm26.2}
If $q\geq x-1$, then $\alpha(G, H)>\beta(G, H)$.
\end{clm}

\begin{proof}
Since $q\geq x-1$, we have $\max G\geq x$ and $\max H\geq x+1$. Let $G<T_1<T_2<\cdots<T_{\lambda}<H$ in $\mathbb{R}_i$. By Claim \ref{clm26.1}, $\max T_j\geq x+1$ for all $j\in [\lambda]$. By (ii) and (iii) of  Proposition \ref{clm23}, $\beta(G, T_1)$, $\beta(T_1, T_2)$, $\cdots$, $\beta(T_{\lambda}, H)$ are  smaller than $d_i$. On the other hand, $\alpha(G, T_1)$, $\alpha(T_1, T_2)$, $\cdots$, $\alpha(T_{\lambda}, H)$ are  $d_i$.
Note that,
\[ \alpha(G, H)=\alpha(G, T_1)+\alpha(T_1, T_2)+\cdots+\alpha(T_{\lambda}, H),\]
and
\[ \beta(G, H)=\beta(G, T_1)+\beta(T_1, T_2)+\cdots+\beta(T_{\lambda}, H). \]
So we conclude that $\alpha(G, H)>\beta(G, H)$.
\end{proof}

By Claim \ref{clm26.2}, we may assume that $k_i+1\leq q \leq x-2$ (Recall that $x\geq k_i+3$).
We will show Lemma \ref{clm3} by induction on $c$.

Let $c=1$. Then $\alpha(F, G)=d_i$. Since $q\leq x-2$, then $\max G\leq x-1<n$. By Proposition \ref{clm23} and the definition of $x$, $\beta (F, G)\geq\sum_{j\ne i}{n-(x-1)\choose k_j-(x-1-k_i)}> d_i$, then $\alpha(F, G)< \beta(F, G)$. So Lemma \ref{clm3} holds for $c=1$. Let $c\geq 2$.
\begin{equation}\label{induction}
\text{ Assume it holds for all $c'\leq c-1$, we will prove that it holds for $c$. }
 \end{equation}
  We will define $c_1, c_2, \dots, c_h$ and  $t_1, t_2, \dots, t_h$, one by one, until $t_1+t_2+\cdots+t_h=x-q-1$, where $h$ is to be determined later.

Let $t_0=0, F_0^+=F$ and $c_0=c$. We determine $c_1$ first.
\begin{clm}\label{clm26.3}
There exists a unique integer $c_1\in[1, c_0-1]$ satisfying the following two conditions.\\
(i) If $F_1$ satisfies $F_0^+\overset{c_1}{\prec} F_1$, then $\alpha(F_0^+, F_1)<\beta(F_0^+, F_1)$;\\
(ii) For any $1\leq j\leq c_0-c_1$ and $F'$ satisfying $F_0^+\overset{c_1+j}{\prec} F'$, we have
$\alpha(F_0^+, F')\geq \beta(F_0^+, F')$.
\end{clm}

\begin{proof}
Let $F'$ be the set  such that $F_0^+\overset{1}{\prec}F'$, i.e., $F_0^+<F'$. Since $q\leq x-2$, $\max F'\leq x-1$. By Proposition \ref{clm23} and the definition of $x$,
\begin{align*}
\beta(F_0^+, F')= \sum_{j\ne i}d_j{n-( q+1)\choose k_j-(q+1-k_i)}> d_i=\alpha(F_0^+, F').
\end{align*}
Note that $F\overset{c}{\prec} G$ and $\alpha(F, G)\geq \beta(F, G)$. Let $c_1$ be the largest integer in $[1, c_0-1]$ satisfying $\alpha(F_0^+, F')< \beta(F_0^+, F')$ for $F'$ satisfying $F_0^+\overset{c_1}{\prec}F'$. Then $c_1$ satisfies both (i) and (ii).
\end{proof}

Define $\mathcal{F}_0^+$ to be the $c_1$-sequential family that range from $q$ to $n$ with $F_0^+$ as it's first member. Since $c_1<c$, by induction hypothesis and the definition of down-up family, we can see that $\mathcal{F}_0^+$  is a down-up family. Let $t_1=:d_{\mathcal{F}_0^+}^{\downarrow}$. Clearly, $1\leq t_1\leq x-q-1$ (in view of Claim \ref{clm26.2}). If $t_1=x-q-1$, then we stop and $h=1$. Otherwise, if $t_1\leq x-q-2$, then we continue to  find $c_2$ and $t_2$. Before performing the next step, we  give the following definitions.

Let $\mathcal{F}_0^+=:\{F_0^+, F_1, M_2^{(1)}, M_3^{(1)}, \dots, M_{t_1}^{(1)}, M_{t_2}^{(1)}, \dots, G_1\}$, where
$$F_0^+\overset{c_1}{\prec} F_1\overset{c_1}{\prec} M_2^{(1)}\overset{c_1}{\prec} M_3^{(1)}\overset{c_1}{\prec} \cdots\overset{c_1}{\prec} M_{t_1}^{(1)}\overset{c_1}{\prec} M_{t_1+1}^{(1)}\overset{c_1}{\prec} \cdots \overset{c_1}{\prec}G_1.$$
Let $M_1^{(1)}=:F_1$ and $\max G_1=n$. Actually, $\max M_{t_1}^{(1)}=q+t_1$.

Since $d_{\mathcal{F}_0^+}^{\downarrow}=t_1$ and $\alpha(M_{t_1}^{(1)}, M_{t_1+1}^{(1)})\geq \beta(M_{t_1}^{(1)}, M_{t_1+1}^{(1)})$,
 then we can define $F_1^+, F_2^+, \dots, F_{t_1}^+, F_2, F_3, \dots, F_{t_1}$ as follows:
$$F_0^+\overset{c_1+1}{\prec} F_1^+\overset{c_1+1}{\prec} F_2^+\overset{c_1+1}{\prec} \cdots\overset{c_1+1}{\prec} F_{t_1}^+,$$
and $$F_1^+\overset{c_1}{\prec} F_2,\,\, F_2^+\overset{c_1}{\prec} F_3,\,\, \dots,\, \,F_{t_1-1}^+\overset{c_1}{\prec} F_{t_1}.$$

For $0\leq k \leq h$, let $a_k:=\sum_{j=0}^kt_j$. Let $2\leq p\leq h$. Assume that $c_1, c_2, \dots, c_{p-1}$ and $t_1, t_2, \dots, t_{p-1}$ have been determined and the condition to terminate is not reached, that is, $t_{p-1}<x-q-a_{p-2}-1$. We next determine $c_p$.

\begin{clm}\label{clm26.4}
There exists a unique integer $c_p$, $1\leq c_p\leq c_{p-1}-1$, satisfying the following two conditions.\\
(i) If $F_{a_{p-1}}^+\overset{c_p}{\prec} F_{a_{p-1}+1}$, then
$\alpha(F_{a_{p-1}}^+, F_{a_{p-1}+1})< \beta(F_{a_{p-1}}^+, F_{a_{p-1}+1})$;\\
(ii) For any $1\leq j\leq c_{p-1}-c_p$ and $F'$ satisfying $F_{a_{p-1}}^+\overset{c_p+j}\prec F'$, we have
$$\alpha(F_{a_{p-1}}^+, F')\geq\beta(F_{a_{p-1}}^+, F').$$
\end{clm}

As Claim \ref{clm26.3}, after the $(p-1)$-th step, we have defined the following family:  \\
\[ \mathcal{F}_{a_{p-2}}^+=\{F_{a_{p-2}}^+, F_{a_{p-2}+1}, M_{a_{p-2}+2}^{(p-1)}, \dots, M_{a_{p-1}}^{(p-1)}, M_{a_{p-1}+1}^{(p-1)}, \dots, G_{a_{p-2}+1}\}, \]
where the sets of $\mathcal{F}_{a_{p-2}}^+$ satisfy
$$F_{a_{p-2}}^+\overset{c_{p-1}}{\prec} F_{a_{p-2}+1}\overset{c_{p-1}}{\prec} M_{a_{p-2}+2}^{(p-1)}\overset{c_{p-1}}{\prec} \cdots\overset{c_{p-1}}{\prec}
M_{a_{p-1}}^{(p-1)}\overset{c_{p-1}}{\prec}
M_{a_{p-1}+1}^{(p-1)}\overset{c_{p-1}}{\prec} \cdots
\overset{c_{p-1}}{\prec} G_{a_{p-2}+1}.$$
Define $F_{a_{p-2}+1}^+, \,F_{a_{p-2}+2}^+,\, \dots, \,F_{a_{p-1}}^+, \,F_{a_{p-2}+1}, \,F_{a_{p-2}+2}, \,\dots, \,F_{a_{p-1}}$ as follows:
$$F_{a_{p-2}}^+\overset{c_{p-1}+1}{\prec}F_{a_{p-2}+1}^+\overset{c_{p-1}+1}{\prec} F_{a_{p-2}+2}^+\overset{c_{p-1}+1}{\prec} \cdots\overset{c_{p-1}+1}{\prec} F_{a_{p-1}}^+,$$
and
$$F_{a_{p-2}}^+\overset{c_{p-1}}{\prec} F_{a_{p-2}+1},\,\, F_{a_{p-2}+1}^+\overset{c_{p-1}}{\prec} F_{a_{p-2}+2},\,\, \dots,\,\, F_{a_{p-1}-1}^+\overset{c_{p-1}}{\prec} F_{a_{p-1}}.$$

\begin{proof}[Proof of Claim \ref{clm26.4}]
First, we can see that $c_{p-1}\geq 2$. Since if not, that is, $c_{p-1}=1$, then $t_{p-1}=d_{\mathcal{F}_{a_{p-2}}^+}^{\downarrow}=x-q-a_{p-2}-1$. On the other hand, since $p-1<h$, we have $t_{p-1}<x-q-a_{p-2}-1$, a contradiction. Let $F'$ be the set satisfying $F_{a_{p-1}}^+\overset{1}{\prec} F'$. Then $\alpha(F_{a_{p-1}}^+, F')< \beta(F_{a_{p-1}}^+, F')$ by Proposition \ref{clm23}. Let $F'$ be the set satisfying $F_{a_{p-1}}^+\overset{c_{p-1}}{\prec} F'$. Since $M_{a_{p-1}}^{(p-1)}\overset{c_{p-1}}{\prec}M_{a_{p-1}+1}^{(p-1)}$ and $\max F_{a_{p-1}}^+=\max M_{a_{p-1}}^{(p-1)}=q+a_{p-1}$, by Claim \ref{clm20},
\begin{align*}
\alpha(F_{a_{p-1}}^+, F')
=\alpha(M_{a_{p-1}}^{(p-1)}, M_{a_{p-1}+1}^{(p-1)})
\geq\beta(M_{a_{p-1}}^{(p-1)}, M_{a_{p-1}+1}^{(p-1)})
=\beta(F_{a_{p-1}}^+, F').
\end{align*}
Let $c_p$ be the maximum integer in $[1, c_{p-1}-1]$ such that if $F_{a_{p-1}}^+\overset{c_p}{\prec} F_{a_{p-1}+1}$, then
$$\alpha(F_{a_{p-1}}^+, F_{a_{p-1}+1})< \beta(F_{a_{p-1}}^+, F_{a_{p-1}+1}).$$
Then $c_p$ satisfies both $(i)$ and $(ii)$.
\end{proof}

After $h$ steps, we can get $x-q$ sets, namely, $F_j$ for $ 1\leq j\le x-q$. We also defined $G_1, G_{a_1+1}, G_{a_2+1}, \dots, G_{a_{h-1}+1}$ as $F_1\overset{c_1}{\longrightarrow}G_1, F_{a_1+1}\overset{c_2}{\longrightarrow}G_{a_1+1}, \dots, F_{a_{h-1}+1}\overset{c_h}{\longrightarrow}G_{a_{h-1}+1}$. For each $1\leq j\leq h$ and $a_{j-1}+1\leq p\leq a_j$, we now define $G_p$ as $F_p \overset{c_j}{\longrightarrow}G_p$.

\begin{clm}\label{clm26.5}
If $c_1+1=c$, then $\alpha(G, H)>\beta(G, H)$.
\end{clm}

\begin{proof}
If $c_1+1=c$, then $F_1^+=G$. Since $F_1\overset{c_1}{\longrightarrow}G_1$, we have $G_1< G$. Then
$$\alpha(F, G)=\alpha(F, F_1)+\alpha(F_1, G_1)+\alpha(G_1, G)$$
and
$$\beta(F, G)=\beta(F, F_1)+\beta(F_1, G_1)+\beta(G_1, G).$$
By the choice of $c_1$ and $t_1\geq 1$, we have $\alpha(F, F_1)< \beta(F, F_1)$. Due to $\max G=q+1$, applying Proposition \ref{clm23}, we have
$$\beta(G_1, G)=\sum_{j\ne i}d_j{n-(q+1)\choose k_j-(q+1-k_i)}.$$
Since $\alpha(F, G)\geq \beta(F, G)$, we get $\alpha(F_1, G_1)+d_i>\beta(F_1, G_1)+\beta(G_1, G)$. Let $\widetilde{G}$ be the set satisfying $\widetilde{G}<H$. Then $G\overset{c_1}{\longrightarrow}\widetilde{G}$. Since $\max G=\max F_1$ and $\max G_1=\max \widetilde{G}$,  by Corollary \ref{coro22},
we obtain $\alpha(F_1, G_1)=\alpha(G, \widetilde{G})$ and $\beta(F_1, G_1)=\beta(G, \widetilde{G})$. Due to $\max H=q+2$, applying Proposition \ref{clm23}, we have
$$\beta(\widetilde{G}, H)=\sum_{j\ne i}d_j{n-(q+2)\choose k_j-(q+2-k_i)}<\beta(G_1, G).$$ So
\begin{align*}
\alpha(G, H)
&=\alpha(G, \widetilde{G})+\alpha(\widetilde{G}, H)\\
&=\alpha(F_1, G_1)+d_i\\
&>\beta(F_1, G_1)+\beta(G_1, G)\\
&>\beta(G, \widetilde{G})+\beta(\widetilde{G}, H)\\
&=\beta(G, H).
\qedhere
\end{align*}
As desired.
\end{proof}
By Claim \ref{clm26.5}, we may assume that $c_1\leq c-2$.
\begin{clm}\label{clm26.6}
Let $0\leq p\leq x-q-2$. Then
$\alpha(F_p^+, F_{p+1}^+)\geq\beta(F_p^+, F_{p+1}^+)$ and $\alpha(F_p^+, F_{p+1})<\beta(F_p^+, F_{p+1}).$
\end{clm}

\begin{proof}
Without loss of generality, assume that $a_j\leq p\leq a_{j+1}-1$ for some $0\leq j\leq h-1$. We next consider the family $\mathcal{F}_{a_j}^+$. Recall that
$$\mathcal{F}_{a_j}^+=\{F_{a_j}^+, F_{a_j+1}, M_{a_j+2}^{(j+1)}, \dots, M_{a_{j+1}}^{(j+1)}, M_{a_{j+1}+1}^{(j+1)}, \dots, G_{a_j+1}\},$$ where
$$F_{a_j}^+\overset{c_{j+1}}{\prec} F_{a_j+1}\overset{c_{j+1}}{\prec} M_{a_j+2}^{(j+1)}\overset{c_{j+1}}{\prec} \cdots\overset{c_{j+1}}{\prec}
M_{a_{j+1}}^{(j+1)}\overset{c_{j+1}}{\prec}
M_{a_{j+1}+1}^{(j+1)}\overset{c_{j+1}}{\prec} \dots
\overset{c_{j+1}}{\prec} G_{a_j+1},$$
and $\max G_{a_j+1}=n.$ We also have the following relations
$$F_{a_j}^+\overset{c_{j+1}+1}{\prec}F_{a_j+1}^+\overset{c_{j+1}+1}{\prec} F_{a_j+2}^+\overset{c_{j+1}+1}{\prec} \cdots\overset{c_{j+1}+1}{\prec} F_{a_{j+1}}^+,$$
$$F_{a_j}^+\overset{c_{j+1}}{\prec} F_{a_j+1},\, F_{a_j+1}^+\overset{c_{j+1}}{\prec} F_{a_j+2},\, \dots,\, F_{a_{j+1}-1}^+\overset{c_{j+1}}{\prec} F_{a_{j+1}}.$$

By the choice of $c_{j+1}$, we have $\alpha(F_{a_j}^+, F_{a_j+1}^+)\geq\beta(F_{a_j}^+, F_{a_j+1}^+)$. Since $F_{a_j}^+\overset{c_{j+1}+1}{\prec}F_{a_j+1}^+$ and $c_{j+1}+1\leq c_1+1< c$, by induction hypothesis (see (\ref{induction})), we have
\begin{equation}\label{eq**}
\alpha(F_{p}^+, F_{p+1}^+)\geq\beta(F_{p}^+, F_{p+1}^+).
\end{equation}
By the definition of $t_{j+1}$, since $F_{a_j}^+\overset{c_{j+1}}{\prec} F_{a_j+1}$,
 we get
\begin{align} \label{eq***}
&\alpha(F_{a_j}^+,  F_{a_j+1})< \beta(F_{a_j}^+,  F_{a_j+1}), \\
&\alpha(F_{a_j+1}, M_{a_j+2}^{(j+1)})< \beta(F_{a_j+1}, M_{a_j+2}^{(j+1)}).\nonumber
\end{align}
Moreover, for each $a_j+2\leq u\leq a_{j+1}-1$, we get
\begin{equation}\label{eq*}
\alpha(M_{u}^{(j+1)}, M_{u+1}^{(j+1)})< \beta(M_{u}^{(j+1)}, M_{u+1}^{(j+1)}).
\end{equation}
Thus Claim \ref{clm26.6} holds for $p=a_j$.

Next we consider $a_j+1\leq p\leq a_{j+1}-1$.
Note that $F_{a_j+1}\overset{c_{j+1}}{\prec} M_{a_j+2}^{(j+1)}, \,F_{a_j+1}^+\overset{c_{j+1}}{\prec} F_{a_j+2}$ and $\max F_{a_j+1}=\max F_{a_j+1}^+$.
Additionally, for $a_j+2\leq p\leq a_{j+1}-1$, we have $M_p^{(j+1)}\overset{c_{j+1}}{\prec}M_{p+1}^{(j+1)}, F_p^+\overset{c_{j+1}}{\prec}F_{p+1}$ and $ \max M_p^{(j+1)}=\max F_p^+$. So Claim \ref{clm20} yields
$$\alpha (F_p^+, F_{p+1})=\alpha(M_p^{(j+1)}, M_{p+1}^{(j+1)})
~~\text{and}~~
\beta (F_p^+, F_{p+1})=\beta(M_p^{(j+1)}, M_{p+1}^{(j+1)}). $$
Hence, for each $a_j+1\leq p\leq a_{j+1}-1$, by (\ref{eq*}), we conclude that
\begin{equation}\label{eq****}
\alpha(F_p^+, F_{p+1})< \beta(F_p^+, F_{p+1}).
\end{equation}
The proof of Claim \ref{clm26.6} is complete.
\end{proof}

\begin{clm}\label{clm26.7}
$\max F_p=\max F_p^+=q+p$ for all $1\leq p\leq x-q-1$.
\end{clm}

\begin{proof}
Let $a_j+1\leq p\leq a_{j+1}$ for some $0\leq j\leq h-1$. Then $F_{p-1}^+\overset{c_{j+1}}{\prec} F_p$ and $F_{p-1}^+\overset{c_{j+1}+1}{\prec} F_p$, so
\begin{equation}\label{eq10}
\max F_p=\max F_p^+.
\end{equation}
We next  prove that $\max F_p=q+p$. For $j=0$, then $1\leq p\leq t_1$. Recall that $F_0^+\overset{c_1}{\prec}F_1, F_1^+\overset{c_1}{\prec}F_2, \dots, F_{t_1-1}^+\overset{c_1}{\prec}F_{t_1}$. By (\ref{eq10}), $\max F_0^+=q$ implies $\max F_p=q+p,$ as desired. Assume it holds for all $j'\leq j-1$, we want to prove it holds for $j$. Recall that $F_{a_j}^+\overset{c_{j+1}}{\prec}F_{a_j+1}, \,F_{a_j+1}^+\overset{c_{j+1}}{\prec}F_{a_j+2}, \,\dots, \,F_{a_{j+1}-1}^+\overset{c_{j+1}}{\prec}F_{a_{j+1}}$. By induction hypothesis, $\max F_{a_j}^+=q+a_j$, then $\max F_{a_j+1}=q+a_j+1, \dots, \max F_{a_{j+1}}=q+a_{j+1}$, as desired.
\end{proof}

\begin{clm}\label{clm26.9}
Let $1\leq p\leq x-q-1$. Then $\alpha(F_p, F_p^+)-\beta(F_p, G_p)>\sum_{j\ne i}d_j{n-(q+p)\choose k_j-(q+p-k_i)}$.
\end{clm}

\begin{proof}
By Claim \ref{clm26.7}, $\max F_p^+=q+p$. By our definition of $G_p$, $G_p<F_p^+$. Applying Proposition \ref{clm23}, we get
$$
\beta(G_p, F_p^+)=\sum_{j\ne i}d_j{n-(q+p)\choose k_j-(q+p-k_i)}.$$
By Claim \ref{clm26.6}, $\alpha(F_{p-1}^+, F_p^+)\geq \beta(F_{p-1}^+, F_p^+)$ and $\alpha(F_{p-1}^+, F_p)<\beta(F_{p-1}^+, F_p)$. On the other hand, we have
$$\alpha(F_{p-1}^+, F_p^+)=\alpha(F_{p-1}^+, F_p)+\alpha(F_p, F_p^+),$$ and
\begin{align*}
\beta(F_{p-1}^+, F_p^+)
&=\beta(F_{p-1}^+, F_p)+\beta(F_p, G_p)+\beta(G_p, F_p^+)\\
&=\beta(F_{p-1}^+, F_p)+\beta(F_p, G_p)+\sum_{j\ne i}d_j{n-(q+p)\choose k_j-(q+p-k_i)}.
\end{align*}
Thus $\alpha(F_p, F_p^+)-\beta(F_p, G_p)>\sum_{j\ne i}d_j{n-(q+p)\choose k_j-(q+p-k_i)}$.
\end{proof}

Define $H_p$ and $J_p$ for each $1\leq p\leq x-q$, that is, $a_0\leq p\leq a_h+1$ as follows.
$$
J_1=G\overset{c_1+1}{\prec}J_2\overset{c_1+1}{\prec}\cdots\overset{c_1+1}{\prec}
J_{a_1}\overset{c_2+1}{\prec}J_{a_1+1}\overset{c_2+1}{\prec}\cdots\overset{c_2+1}{\prec}J_{a_2}\overset{c_3+1}{\prec}
\cdots\overset{c_h+1}{\prec}J_{a_h}\overset{c_h+1}{\prec}J_{a_h+1},
$$
where the last set $J_{a_h+1}$ exists since Claim \ref{clm26.7} implies that $\max J_{a_h+1}=q+x-q-1+1\leq n$. Let $H_p$ be the set such that $H_p<J_{p+1}$ in $\mathbb{R}_i$.

By the definition of $J_p, 1\leq p\leq x-q$, we get $\max J_p=q+p$. Proposition \ref{clm23} gives
\begin{equation}\label{eq11}
\beta(H_p, J_{p+1})=\sum_{j\ne i}d_j{n-(q+p+1)\choose k_j-(q+p+1-k_i)}.
\end{equation}

\begin{clm}\label{clm26.10}
Let $1\leq p\leq x-q$. Then $\alpha(J_p, H_p)=\alpha(F_p, G_p)$ and $\beta(J_p, H_p)=\beta(F_p, G_p)$.
\end{clm}

\begin{proof}
By Claim \ref{clm26.7} and $\max J_p=q+p$, we have $\max J_p=\max F_p$. Trivially, $\max H_p=\max G_p=n$. By our definition, $J_p$ and $H_p$ are $c_y$-sequential for some $y$, and $F_p$ and $G_p$ are $c_y$-sequential as well. It follows from Corollary \ref{coro22} that $\alpha(J_p, H_p)=\alpha(F_p, G_p)$ and $\beta(J_p, H_p)=\beta(F_p, G_p)$.
\end{proof}

Accordingly,
\begin{align}\label{eq12}
\alpha(J_p, J_{p+1})-\beta(J_p, H_p)
&=\alpha(J_p, H_p)+d_i-\beta(F_p, G_p) \nonumber \\
&=\alpha(F_p, G_p)+d_i-\beta(F_p, G_p)\nonumber \\
&=\alpha(F_p, F_p^+)-\beta(F_p, G_p)\nonumber \\
&>\sum_{j\ne i}d_j{n-(q+p)\choose k_j-(q+p-k_i)},
\end{align}
where the first and second equalities hold by Claim \ref{clm26.10} and the last inequality holds by Claim \ref{clm26.9}.
Furthermore,
\begin{align}\label{eq13}
\alpha(J_p, J_{p+1})-\beta(J_p, J_{p+1})
&=\alpha(J_p, J_{p+1})-\beta(J_p, H_p)-\beta(H_p, J_{p+1})\nonumber \\
&=\alpha(J_p, J_{p+1})-\beta(J_p, H_p)-\sum_{j\ne i}d_j{n-(q+p+1)\choose k_j-(q+p+1-k_i)}\nonumber \\
&>\sum_{j\ne i}d_j\left[{n-(q+p)\choose k_j-(q+p-k_i)}-{n-(q+p+1)\choose k_j-(q+p+1-k_i)}\right],
\end{align}
where the second equality holds by (\ref{eq11}) and the last inequality holds by (\ref{eq12}).

Let $J_{n-q}$ be the set such that $J_{a_h+1}\overset{c_h+1}{\longrightarrow}J_{n-q}$. In particular, if $n=x+1$, then $J_{n-q}=J_{a_h+1}$.

\begin{clm}\label{clm26.11}
Let $1\leq p\leq x-q-1$. Then $$\alpha(J_p, J_{n-q})-\beta(J_p, J_{n-q})>\sum_{j\ne i}d_j{n-(q+p)\choose k_j-(q+p-k_i)}.$$
\end{clm}

\begin{proof}
Without loss of generality, let $a_{j-1}+1\leq p\leq a_j$ for some $1\leq j\leq h$. By our definition,
\begin{equation}\label{eq14}
J_p\overset{c_j+1}{\prec}J_{p+1}\overset{c_j+1}{\prec}\cdots\overset{c_j+1}{\prec}J_{a_j}
\overset{c_{j+1}+1}{\prec}J_{a_j+1}\overset{c_{j+1}+1}{\prec}\cdots\overset{c_h+1}{\prec}
J_{a_h}\overset{c_h+1}{\prec}J_{a_h+1}\overset{c_h+1}{\prec}\cdots\overset{c_h+1}{\prec}J_{n-q}.
\end{equation}
Let $T_1, T_2, \dots, T_Y\in \mathbb{R}_i$ be the sets such that
$J_{a_h}<T_1< T_2< \dots <T_Y<J_{n-q}.$ By Claim \ref{clm26.1}, $\max T_j\geq \max J_{a_h}+1=x-q$ holds for all $j\in [Y]$. By Proposition \ref{clm23},
and the definition of $x$,
\begin{align*}
\beta(J_{a_h}, J_{n-q})&=\beta(J_{a_h}, T_1)+\beta(T_1, T_2)+\cdots+\beta(T_Y, J_{n-q})\\
&<\alpha(J_{a_h}, T_1)+\alpha(T_1, T_2)+\cdots+\alpha(T_Y, J_{n-q})\\
&=\alpha(J_{a_h}, J_{n-q}).
\end{align*}
Then (\ref{eq13}) and (\ref{eq14}) give
\begin{align*}
\alpha(J_p, J_{n-q})-\beta(J_p, J_{n-q})
&=\alpha(J_p, J_{p+1})+\cdots+\alpha(J_{a_h-1}, J_{a_h})+\alpha(J_{a_h}, J_{n-q})\\
&\quad-\beta(J_p, J_{p+1})-\cdots-\beta(J_{a_h-1}, J_{a_h})-\beta(J_{a_h}, J_{n-q})\\
&>\sum_{j\ne i}d_j{n-(q+p)\choose k_j-(q+p-k_i)}.
\end{align*}
\end{proof}

It is easy to see that $1\leq c_1+1-c_h<c-c_h$. If $c-c_h=2$, then since $c-c_1\geq 2$, we have $h=1$ and $c_1+1=c-1$. By (\ref{eq14}), $J_1\overset{c_1+1}{\longrightarrow}J_{n-q}$ and $J_{n-q}<H$. By Proposition \ref{clm23} and $\max H=q+2$,
\begin{equation}\label{0}
\beta(J_{n-q}, H)=\sum_{j\ne i}d_j{n-(q+2)\choose k_j-(q+2-k_i)}.
\end{equation}

Then by Claim \ref{clm26.11},
\begin{align*}
\alpha(G, H)-\beta(G, H)
&=\alpha(J_1, H)-\beta(J_1, H)\\
&=\alpha(J_1, J_{n-q})+\alpha(J_{n-q}, H)-\beta(J_1, J_{n-q})-\beta(J_{n-q}, H)\\
&>\sum_{j\ne i}d_j\left[{n-(q+1)\choose k_j-(q+1-k_i)}-{n-(q+2)\choose k_j-(q+2-k_i)}\right]+d_i\\
&>0,
\end{align*}
where the first inequality holds by Claim \ref{clm26.11}, equation (\ref{0}) and $\alpha(J_{n-q}, H)=d_i$.
As desired.

Next we assume that $c-c_h\geq 3$.

Since $c_h<c_{h-1}<\cdots<c_1<c$ and $c-c_h>2$, we may define sequential families $\mathcal{F}_p$ for each  $1\leq p\leq c-c_h-2$, as follows.
Let $c_d-c_h+1\leq p\leq c_{d-1}-c_h$ for some $d=2, \dots, h$ or $c_1-c_h+1\leq p\leq c-c_h-2$. We define
$$\mathcal{F}_p: J_{a_{d-1}+1}\overset{c_h+1+p}{\prec}J_{a_{d-1}+2}^{(p)}\overset{c_h+1+p}{\prec}\cdots\overset{c_h+1+p}{\prec}J_{a_d}^{(p)}\overset{c_h+1+p}{\prec}J_{a_d+1}^{(p)}\overset{c_h+1+p}{\prec}\cdots\overset{c_h+1+p}{\prec}J_{n-q}^{(p)}.$$
By our definition, for any $J_j^{(p)}\in \mathcal{F}_p$, we get
\begin{equation}\label{eq14*}
\max J_j^{(p)}=q+j.
\end{equation}
Knowing that $J_{a_{h-1}+1}\overset{c_h+1}{\prec}\cdots\overset{c_h+1}{\prec}J_{a_h}\overset{c_h+1}{\prec}J_{a_h+1}\overset{c_h+1}{\longrightarrow}J_{n-q}$,
$J_{a_{h-1}+1}\overset{c_h+2}{\prec}J_{a_{h-1}+2}^{(1)}$ and $\alpha(J_{n-q}, J_{a_{h-1}+2}^{(1)})=d_i$, we also denote $J_{n-q, 1}^{(1)}, J_{n-q, 2}^{(1)}, \dots, J_{{n-q}, t_h}^{(1)}$ as follows
\begin{align*}
&J_{a_{h-1}+2}^{(1)}\overset{c_h+1}{\longrightarrow}J_{n-q, 1}^{(1)},\,\, i.e., \,\,J_{{n-q}, 1}^{(1)}<J_{a_{h-1}+3}^{(1)};\\
&J_{a_{h-1}+3}^{(1)}\overset{c_h+1}{\longrightarrow}J_{n-q, 2}^{(1)},\,\, i.e., \,\,J_{{n-q}, 2}^{(1)}<J_{a_{h-1}+4}^{(1)};\\
&\quad\vdots\\
&J_{a_h+1}^{(1)}\overset{c_h+1}{\longrightarrow}J_{n-q, t_h}^{(1)},\,\, i.e.,\,\,J_{{n-q}, t_h}^{(1)}<J_{a_h+2}^{(1)}.
\end{align*}

Consequently,
\begin{align*}
&\alpha(J_{a_{h-1}+1}, J_{n-q}^{(1)})-\beta(J_{a_{h-1}+1}, J_{n-q}^{(1)})\\
&=\alpha(J_{a_{h-1}+1}, J_{n-q})+\alpha(J_{n-q}, J_{a_{h-1}+2}^{(1)})+\alpha(J_{a_{h-1}+2}^{(1)}, J_{{n-q}, 1}^{(1)})+\cdots\\
&\quad+\alpha(J_{{n-q}, t_h-1}^{(1)}, J_{a_h+1}^{(1)})+\alpha(J_{a_h+1}^{(1)}, J_{n-q}^{(1)})- \beta(J_{a_{h-1}+1}, J_{n-q})\\
&\quad-\beta(J_{n-q}, J_{a_{h-1}+2}^{(1)})
-\beta(J_{a_{h-1}+2}^{(1)}, J_{{n-q}, 1}^{(1)})-\cdots\\
&\quad-\beta(J_{{n-q}, t_h-1}^{(1)}, J_{a_h+1}^{(1)})
-\beta(J_{a_h+1}^{(1)}, J_{n-q}^{(1)}).
\end{align*}

Applying Corollary \ref{coro22}, we get
\begin{align*}
&\alpha(J_{a_{h-1}+2}^{(1)}, J_{{n-q}, 1}^{(1)})=\alpha(J_{a_{h-1}+2}, J_{n-q}), \,\beta(J_{a_{h-1}+2}^{(1)}, J_{{n-q}, 1}^{(1)})=\beta(J_{a_{h-1}+2}, J_{n-q}),\\
&\alpha(J_{a_{h-1}+3}^{(1)}, J_{{n-q}, 2}^{(1)})=\alpha(J_{a_{h-1}+3}, J_{n-q}), \,\beta(J_{a_{h-1}+3}^{(1)}, J_{{n-q}, 2}^{(1)})=\beta(J_{a_{h-1}+3}, J_{n-q}),\\
&\quad \vdots\\
&\alpha(J_{a_h+1}^{(1)}, J_{{n-q}, t_h}^{(1)})=\alpha(J_{a_h+1}, J_{n-q}), \,\beta(J_{a_h+1}^{(1)}, J_{{n-q}, t_h}^{(1)})=\beta(J_{a_h+1}, J_{n-q}).
\end{align*}

By Proposition \ref{clm23} and (\ref{eq14*}),
\begin{align*}
&\beta(J_{n-q}, J_{a_{h-1}+2}^{(1)})=\sum_{j\ne i}d_j{n-(a_{h-1}+2+q)\choose k_j-(a_{h-1}+2+q-k_i)}, \\
&\beta(J_{{n-q}, 1}^{(1)}, J_{a_{h-1}+3}^{(1)})=\sum_{j\ne i}d_j{n-(a_{h-1}+3+q)\choose k_j-(a_{h-1}+3+q-k_i)}, \\
&\quad \vdots\\
& \beta(J_{{n-q}, t_h-1}^{(1)}, J_{a_h+1}^{(1)})=\beta(J_{a_h+1}^{(1)}, J_{n-q}^{(1)})<d_i.
\end{align*}
Then by Claim \ref{clm26.11}, we have
\begin{align}\label{eq15}
&\alpha(J_{a_{h-1}+1}, J_{n-q})-\beta(J_{a_{h-1}+1}, J_{n-q})\\
&>\sum_{j\ne i}d_j\left[{n-(a_{h-1}+1+q)\choose k_j-(a_{h-1}+1+q-k_i)}-{n-(a_{h-1}+2+q)\choose k_j-(a_{h-1}+2+q-k_i)}\right.\nonumber\\
&\quad \left.+{n-(a_{h-1}+2+q)\choose k_j-(a_{h-1}+2+q-k_i)}-\cdots+{n-(a_h+1+q)\choose k_j-(a_h+1+q-k_i)}\right]\nonumber\\
&=\sum_{j\ne i}d_j{n-(a_{h-1}+1+q)\choose k_j-(a_{h-1}+1+q-k_i)}.
\end{align}

Using the same argument, we get
\begin{align}\label{eq16}
&\alpha(J_{a_{h-1}+2}^{(1)}, J_{n-q}^{(1)})-\beta(J_{a_{h-1}+2}^{(1)}, J_{n-q}^{(1)})>\sum_{j\ne i}d_j{n-(a_{h-1}+2+q)\choose k_j-(a_{h-1}+2+q-k_i)},\\
&\quad\vdots \nonumber\\ \label{eq17}
&\alpha(J_{a_h}^{(1)}, J_{n-q}^{(1)})-\beta(J_{a_h}^{(1)}, J_{n-q}^{(1)})>\sum_{j\ne i}d_j{n-(a_h+q)\choose k_j-(a_h+q-k_i)},\\ \label{eq18}
&\alpha(J_{a_h+1}^{(1)}, J_{n-q}^{(1)})-\beta(J_{a_h+1}^{(1)}, J_{n-q}^{(1)})>0.
\end{align}

\begin{clm}\label{clm26.12}
Let $1\leq k\leq c-c_h-2$ and $D\in \mathcal{F}_k$ with $\max D=p+	q$. Then
$$\alpha(D, J_{n-q}^{(k)})-\beta(D, J_{n-q}^{(k)})>\sum_{j\ne i}d_j{n-(p+q)\choose k_j-(p+q-k_i)}.$$
\end{clm}

\begin{proof}
By induction on $k$. For $k=1$, following from (\ref{eq15})-- (\ref{eq18}), we are done. Assume that it holds for $\mathcal{F}_j, j\in [1, c-c_h-3]$, we want to prove it holds for $\mathcal{F}_{j+1}$. Define $\widetilde{J}_2^{(j)}, \dots,
\widetilde{J}_{t_1}^{(j)}, \\\widetilde{J}_{t_1+1}^{(j)},
\dots, \widetilde{J}_{n-q}^{(j)}$ as follows:
$\widetilde{J}_p^{(j)}<J_p^{(j)}, p=2,\dots, n$.
Note that $\widetilde{J}_2^{(j)}=J_{n-q}^{(j-1)}$. By induction hypothesis, and $\max J_1=q+1$, we have
\begin{align}
\alpha(J_1, \widetilde{J}_2^{(j)})-\beta(J_1, \widetilde{J}_2^{(j)})
&=\alpha(J_1, J_{n-q}^{(j-1)})-\beta(J_1, J_{n-q}^{(j-1)})\nonumber \\ \label{eq19}
&>\sum_{j\ne i}d_j{n-(q+1)\choose k_j-(q+1-k_i)}.
\end{align}
And for $2\leq p\leq n-q$, we have
\begin{equation}\label{eq20}
\alpha(J_p^{(j-1)}, \widetilde{J}_p^{(j)})-\beta(J_p^{(j-1)}, \widetilde{J}_p^{(j)})>\sum_{j\ne i}d_j{n-(q+p)\choose k_j-(q+p-k_i)}.
\end{equation}
Recall that for $2\leq p\leq n-q-1$, we have
$$J_p^{(j)}\overset{c_h+j}{\longrightarrow}\widetilde{J}_{p+1}^{(j)}, \,\,J_p^{(j-1)}\overset{c_h+j}{\longrightarrow}J_{n-q}^{(j-1)}$$
and
$$\max J_p^{(j)}=\max J_p^{(j-1)}=q+p,\,\, \max \widetilde{J}_{p+1}^{(j)}=\max J_{n-q}^{(j-1)}=n.$$
Applying Corollary \ref{coro22}, we get
$$\alpha(J_p^{(j)}, {J}_{p+1}^{(j)})=\alpha(J_p^{(j-1)}, J_{n-q}^{(j-1)})$$
and
$$\beta(J_p^{(j)}, {J}_{p+1}^{(j)})=\beta(J_p^{(j-1)}, J_{n-q}^{(j-1)}).$$

By Proposition \ref{clm23} and inequalities (\ref{eq19}), (\ref{eq20}), if $2\leq p\leq n-q-1$, then
\begin{align*}
&\alpha(J_p^{(c_h+j-1)}, J_{n-q}^{(c_h+j-1)})-\beta(J_p^{(c_h+j-1)}, J_{n-q}^{(c_h+j-1)})\\
&=\alpha(J_p^{(c_h+j-1)}, \widetilde{J}_{p+1}^{(c_h+j-1)})+\alpha(\widetilde{J}_{p+1}^{(c_h+j-1)}, J_{p+1}^{(c_h+j-1)})+\alpha(J_{p+1}^{(c_h+j-1)}, \widetilde{J}_{p+2}^{(c_h+j-1)})+\cdots\\
&\quad+\alpha(\widetilde{J}_{n-q}^{(c_h+j-1)}, J_{n-q}^{(c_h+j-1)})-\beta(J_p^{(c_h+j-1)}, \widetilde{J}_{p+1}^{(c_h+j-1)})-\beta(\widetilde{J}_{p+1}^{(c_h+j-1)}, J_{p+1}^{(c_h+j-1)})\\
&\quad-\beta(J_{p+1}^{(c_h+j-1)}, \widetilde{J}_{p+2}^{(c_h+j-1)})
-\cdots-\beta(\widetilde{J}_{n-q}^{(c_h+j-1)}, J_{n-q}^{(c_h+j-1)})\\
&>\sum_{j\ne i}d_j\left[ {n-(q+p)\choose k_j-(q+p-k_i)}-{n-(q+p+1)\choose k_j-(q+p+1-k_i)}+{n-(q+p+1)\choose k_j-(q+p+1-k_i)}\right.\\
&\quad \left.-\cdots+{n-(x)\choose k_j-(x-k_i)}-{n-(x+1)\choose k_j-(x+1-k_i)} \right]\\
&=\sum_{j\ne i}d_j{n-(q+p)\choose k_j-(q+p-k_i)},
\end{align*}
where the second inequality follows from (\ref{eq20}) and Proposition \ref{clm23}.

For $p=1$, by (\ref{eq19}) and using the same argument as above, we get
\begin{equation}\label{eq21}
\alpha(J_1, J_{n-q}^{(c_h+j-1)})-\beta(J_1, J_{n-q}^{(c_h+j-1)})>\sum_{j\ne i}d_j{n-(q+1)\choose k_j-(q+1-k_i)}.
\end{equation}
\end{proof}

Next, we are going to complete the proof of Lemma \ref{clm3}.

Recall that $J_{n-q}^{(c-3)}<H$ and $\max H=q+2, G=J_1$, so
\begin{align*}
\alpha(G, H)-\beta(G, H)
&=\alpha(G, J_{n-q}^{(c-3)})+\alpha(J_{n-q}^{(c-3)}, H)-\beta(G, J_{n-q}^{(c-3)})-\beta(J_{n-q}^{(c-3)}, H)\\
&>\sum_{j\ne i}d_j\left[ {n-(q+1)\choose k_j-(q+1-k_i)}-{n-(q+2)\choose k_j-(q+2-k_i)} \right]+d_i\\
&>0,
\end{align*}
where the second inequality follows from (\ref{eq21}), Proposition \ref{clm23} and $\alpha(J_{n-q}^{(c-3)}, H)=d_i$.
The proof of Lemma \ref{clm3} is complete.

\subsection{Proof of Lemma \ref{clm28}}

Recall that $\mathcal{R}_{i, k}=:\{R\in\mathbb{R}_i: [n-k+1, n]\subset R\},
\mathcal{R}_i(k)=:\{R\setminus [n-k+1, n]: R\in\mathcal{R}_{i, k}\}$ for $k\in [k_i-1]$.
By Remark \ref{coro11} and using the same argument as Claim \ref{clm20}, we have the following claim.

\begin{clm}\label{clm27}
Let $1\leq j\leq k_i-1$ and $1\leq d\leq k_i-j$. Let $F, H, F', H'\in \mathcal{R}_i(j)$ and $F\overset{d}{\prec}H, F'\overset{d}{\prec}H'$. If $\max F=\max F'$, then $\alpha(F, H)=\alpha(F', H')$ and $\beta(F, H)=\beta(F', H')$.
\end{clm}

\begin{clm}\label{0000}
Let $F_1<G_1, F_2<G_2$ in $\mathcal{R}_i(j), j\in [0, k_i-1]$ with $\max G_1=\max G_2$. Then $\alpha(F_1, G_1)=\alpha(F_2, G_2)$ and $\beta(F_1, G_1)=\beta(F_2, G_2)$.
\end{clm}

\begin{proof}
For $j=0$, we can see that $\alpha(F_1, G_1)=\alpha(F_2, G_2)=1$, and then Proposition \ref{clm23} gives $\beta(F_1, G_1)=\beta(F_2, G_2)$. Now assume that $j\geq 1$. Let $F'_1=F_1\sqcup \{n-j+1, \dots, n\}, F'_2=F_2\sqcup \{n-j+1, \dots, n\}, G'_1=G_1\sqcup \{n-j+1, \dots, n\}, G'_2=G_2\sqcup \{n-j+1, \dots, n\},$ then $F'_1, F'_2, G'_1, G'_2\in \mathbb{R}_i$.
Let $H_1$ and $H_2$ be the sets such that $F'_1<H_1$ and $F'_2<H_2$ in $\mathbb{R}_i$. We get $H_1\overset{j}{\longrightarrow}G'_1$ and $H_2\overset{j}{\longrightarrow}G'_2$. By the definitions of $F_1, G_1, F_2$ and $G_2$, we have $\max H_1=\max H_2$ and  $\max G'_1=\max G'_2$. So Corollary \ref{coro22} gives $\alpha(F'_1, G'_1)=\alpha(F'_2, G'_2)$ and $\beta(F'_1, G'_1)=\beta(F'_2, G'_2)$, that is $\alpha(F_1, G_1)=\alpha(F_2, G_2)$ and $\beta(F_1, G_1)=\beta(F_2, G_2)$.
\end{proof}

It's easy to check the following corollary by using a similar argument of Corollary\ref{coro22}.
\begin{corollary}\label{000}
Let $d\in [k_i-j]$ and $F, G, F', G'\in \mathcal{R}_i(j)$. If $F, G$ are $d$-sequential, $F', G'$ are $d$-sequential satisfying $\max F=\max F'$ and $\max G=\max G'$, then $\alpha (F, G)=\alpha (F', G')$ and $\beta (F, G)=\beta (F', G')$.
\end{corollary}

\begin{proof}[Proof for Lemma \ref{clm28}]
We prove Lemma \ref{clm28} by induction on $j$. It holds for $j=0$ by Lemma \ref{clm3}. Suppose it holds for $j\in[0, k_i-2]$, we are going to prove it holds for $j+1$. Let $F, G, H \in \mathcal{R}_i(j+1)$ with $F\overset{c}{\prec}G\overset{c}{\prec}H$ and $\alpha(F, G)
\geq \beta(F, G)$. We are going to apply induction assumption to show $\alpha(G, H)>\beta(G, H)$.
Let $F'=F\sqcup \{\max F+1\}, G'=G\sqcup \{\max G+1\}$ and $H'=H\sqcup \{\max H+1\}$.
Then $F', G', H'\in \mathcal{R}_i(j)$. Moreover, $F'\overset{c+1}{\prec}G'\overset{c+1}{\prec}H'$ in  $\mathcal{R}_i(j)$.

Let $G_1, G_2, H_1, F_1, F_2$ be sets satisfying $G_1<G'<G_2, H_1<H', F'\overset{c}{\prec}F_1$ and $F'<F_2$. Let $\widetilde{F}=F\sqcup\{n-j\}, \widetilde{G}=G\sqcup\{n-j\}$, $\widetilde{H}=H\sqcup\{n-j\}$. Then $\widetilde{F}, \widetilde{G}, \widetilde{H}\in \mathcal{R}_i(j)$.
We can see that
if $c\geq 2$, then
\begin{equation}\label{eq22}
F'<F_2\overset{1}{\longrightarrow} \widetilde{F}< F_1\overset{c}{\longrightarrow} G_1<G'<G_2\overset{1}{\longrightarrow}  \widetilde{G}\, \,\text{and}\,\,G'\overset{c}{\longrightarrow}H_1<H';
\end{equation}
if $c=1$, then
\begin{equation}\label{eq22+}
F'<F_1=F_2\overset{1}{\longrightarrow} \widetilde{F}=G_1<G'<G_2\overset{1}{\longrightarrow}  \widetilde{G}\, \,\text{and}\,\,G'\overset{c}{\longrightarrow}H_1<H'.
\end{equation}

\begin{clm}\label{00000}
$\alpha(F_1, G_1)>\beta(F_1, G_1).$
\end{clm}

\begin{proof}
Suppose on the contrary that $\alpha(F_1, G_1)\leq\beta(F_1, G_1).$
We first consider the case $c\geq 2$.
By (\ref{eq22}),
\begin{align*}
&\alpha(\widetilde{F}, \widetilde{G})=\alpha(\widetilde{F}, F_1)+\alpha(F_1, G_1)+\alpha(G_1, G')+\alpha(G', \widetilde{G}),\\
&\beta(\widetilde{F}, \widetilde{G})=\beta(\widetilde{F}, F_1)+\beta(F_1, G_1)+\beta(G_1, G')+\beta(G', \widetilde{G}).
\end{align*}
Note that $\alpha(F, G)\geq\beta(F, G)$ means $\alpha(\widetilde{F}, \widetilde{G})\geq\beta(\widetilde{F}, \widetilde{G})$. Since $\alpha(F_1, G_1)\leq\beta(F_1, G_1)$, then
\begin{equation}\label{eq23}
\alpha(\widetilde{F}, F_1)+\alpha(G_1, G')+\alpha(G', \widetilde{G})\geq\beta(\widetilde{F}, F_1)+\beta(G_1, G')+\beta(G', \widetilde{G}).
\end{equation}

Note that $\max F_2=\max G'.$ By Claim \ref{0000}, we have $\beta(F', F_2)=\beta(G_1, G')$ and $\alpha(F', F_2)=\alpha(G_1, G')$. Note that $F_2\overset{1}{\longrightarrow} \widetilde{F}, G'\overset{1}{\longrightarrow} \widetilde{G}$, $\max F_2=\max G'$ and $\max  \widetilde{F}=\max  \widetilde{G} $, it follows from Corollary \ref{000} that
$\alpha(F_2, \widetilde{F})=\alpha(G', \widetilde{G})$ and $\beta(F_2, \widetilde{F})=\beta(G', \widetilde{G}).$  Then
\begin{align*}
\alpha(\widetilde{F}, F_1)+\alpha(G_1, G')+\alpha(G', \widetilde{G})
&=\alpha(\widetilde{F}, F_1)+\alpha(F', F_2)+\alpha(F_2, \widetilde{F})=\alpha(F', F_1).
\end{align*}

Similarly, we have
$$\beta(\widetilde{F}, F_1)+\beta(G_1, G')+\beta(G', \widetilde{G})=\beta(F', F_1).$$

So inequality (\ref{eq23}) gives $\alpha(F', F_1)\geq\beta(F', F_1)$.

Note that $F'\overset{c}{\prec}F_1, F_1\overset{c}{\longrightarrow}G_1 \in\mathcal{R}_i(j), c\in [k_i-j]$, by induction hypothesis, $\alpha(F_1, G_1)>\beta(F_1, G_1)$. A contradiction to our assumption.

We next consider the case $c=1$. The proof is quite similar to the above case, for completeness, we write here.
By (\ref{eq22+}),
\begin{align*}
&\alpha(\widetilde{F}, \widetilde{G})=\alpha(G_1, G')+\alpha(G', \widetilde{G}),\\
&\beta(\widetilde{F}, \widetilde{G})=\beta(G_1, G')+\beta(G', \widetilde{G}).
\end{align*}
Note that $\alpha(F, G)\geq\beta(F, G)$ means $\alpha(\widetilde{F}, \widetilde{G})\geq\beta(\widetilde{F}, \widetilde{G})$. Since $\alpha(F_1, G_1)\leq\beta(F_1, G_1)$, then
\begin{equation}\label{eq23+}
\alpha(G_1, G')+\alpha(G', \widetilde{G})\geq\beta(G_1, G')+\beta(G', \widetilde{G}).
\end{equation}

Note that $\max F_2=\max G'.$ By Claim \ref{0000}, we have $\beta(F', F_2)=\beta(G_1, G')$ and $\alpha(F', F_2)=\alpha(G_1, G')$. Note that $F_2\overset{1}{\longrightarrow} \widetilde{F}, G'\overset{1}{\longrightarrow} \widetilde{G}$, $\max F_2=\max G'$ and $\max  \widetilde{F}=\max  \widetilde{G} $, it follows from Corollary \ref{000} that
$\alpha(F_2, \widetilde{F})=\alpha(G', \widetilde{G})$ and $\beta(F_2, \widetilde{F})=\beta(G', \widetilde{G}).$  Then
\begin{align*}
\alpha(G_1, G')+\alpha(G', \widetilde{G})
&=\alpha(F', F_2)+\alpha(F_2, \widetilde{F})=\alpha(F', \widetilde{F}).
\end{align*}

Similarly, we have
$$\beta(\widetilde{F}, F_1)+\beta(G_1, G')+\beta(G', \widetilde{G})=\beta(F',\widetilde{F}).$$
So inequality (\ref{eq23+}) gives $\alpha(F', \widetilde{F})\geq\beta(F', \widetilde{F})$.
Note that
\[
\alpha(F', F_1)=\alpha(F'_1, \widetilde{F})-\alpha(F_1, G_1)
\]
and
\[
\beta(F'_1, F_1)=\beta(F'_1, \widetilde{F})-\beta(F_1, G_1).
\]
Since $\alpha(F_1, G_1)\leq\beta(F_1, G_1)$ and $\alpha(F', \widetilde{F})\geq\beta(F', \widetilde{F})$, we get $\alpha(F', F_1)\geq\beta(F'_1, F_1)$.
Note that $F'\overset{c}{\prec}F_1, F_1\overset{c}{\longrightarrow}G_1 \in\mathcal{R}_i(j), c\in [k_i-j]$, by induction hypothesis, $\alpha(F_1, G_1)>\beta(F_1, G_1)$. A contradiction to our assumption.
\end{proof}

By (\ref{eq22}), we have $G'\overset{c}{\longrightarrow}H_1, F_1\overset{c}{\longrightarrow} G_1,\max G'=\max F_1, \max H_1=\max G_1$, by Corollary \ref{000} and Claim \ref{00000}, we get
\begin{equation}\label{hh}
\alpha(G', H_1)>\beta(G', H_1).
\end{equation}
Since $G'<G_2$ and $H_1<H'$ in $\mathcal{R}_i(j)$, by Claim \ref{0000}, $\alpha(G', G_2)=\alpha(H_1, H')$ and $\beta(G', G_2)=\beta(H_1, H')$. Then $f_i(G_2)<f_i(H')$ following from (\ref{hh}). Recall that $G_2\overset{1}{\longrightarrow}\widetilde{G}$ and $H'\overset{1}{\longrightarrow}\widetilde{H}$. Hence, $f_i(\widetilde{G})<f_i(\widetilde{H})$ by applying Corollary \ref{000}. This implies $\alpha(G, H)>\beta(G, H)$, as desired. The proof of Lemma \ref{clm28} is complete.
\end{proof}

\subsection{Proofs of Lemma \ref{clm29} and Lemma \ref{clm30} }
We only give the proof of Lemma \ref{clm29}, Lemma \ref{clm30} can be proved by the same argument.
\begin{proof}[Proof of Lemma \ref{clm29}]
Since $f(\{2, 3, \dots, j\})\leq f(\{2, 3, \dots, j-1\})$, we have
\begin{equation}\label{eq32}
\alpha(\{2, 3, \dots, j\}, \{2, 3, \dots, j-1\})\geq \beta(\{2, 3, \dots, j\}, \{2, 3, \dots, j-1\}).
\end{equation}
We need the following claim.
\begin{clm}\label{clm29.1}
\begin{align*}
&\alpha(\{2, 3, \dots, j\}, \{2, 3, \dots, j-1\})=\alpha(\{2, 3, \dots, j-1\}, \{2, 3, \dots, j-2, j\}),\\
&\beta(\{2, 3, \dots, j\}, \{2, 3, \dots, j-1\})=\beta(\{2, 3, \dots, j-1\}, \{2, 3, \dots, j-2, j\}).
\end{align*}
\end{clm}

\begin{proof}[Proof of Claim \ref{clm29.1}]
Note that the sets in $\mathcal{L}([n], \{2, 3, \dots, j-1\}, k_i)\setminus \mathcal{L}([n], \{2, 3, \dots, j\}, k_i)$ are the $k_i$-sets containing $\{2, 3, \dots, j-1\}$ but containing neither $\{1\}$ nor $\{j\}$.
Then we can see that
\begin{align*}
&\alpha(\{2, 3, \dots, j\}, \{2, 3, \dots, j-1\})\\
&=d_i\big(|\mathcal{L}([n], \{2, 3, \dots, j-1\}, k_i)|-|\mathcal{L}([n], \{2, 3, \dots, j\}, k_i)|\big)\\
&=d_i{n-j\choose k_i-j+2}.
\end{align*}
Since the sets in $\mathcal{L}([n], \{2, 3, \dots, j-2, j\}, k_i)\setminus \mathcal{L}([n], \{2, 3, \dots, j-2, j-1\}, k_i)$ are the $k_i$-sets containing $\{2, 3, \dots, j-2, j\}$ but containing neither $\{1\}$ nor $\{j-1\}$,
we also get
$$
\alpha(\{2, 3, \dots, j-1\}, \{2, 3, \dots, j-2, j\})=d_i{n-j\choose k_i-j+2}.
$$
So we have
$$\alpha(\{2, 3, \dots, j\}, \{2, 3, \dots, j-1\})=\alpha(\{2, 3, \dots, j-1\}, \{2, 3, \dots, j-2, j\}),$$
\begin{align*}
\beta(\{2, 3, \dots, j\}, \{2, 3, \dots, j-1\})
&=\sum_{p\ne i}d_j\left [ {n-2\choose k_p-2}+\cdots+{n-j\choose k_p-2}-{n-2\choose k_p-2}- \right.\\
&\quad \left.\cdots-{n-(j-1)\choose k_p-2} \right]\\
&=\sum_{p\ne i}d_j{n-j\choose k_p-2},
\end{align*}
and
\begin{align*}
\beta(\{2, 3, \dots, j-1\}, \{2, 3, \dots, j-2, j\})
&=\sum_{p\ne i}d_j\left [ {n-2\choose k_p-2}+\cdots+{n-(j-1)\choose k_p-2} \right.\\
&\quad \left.-{n-2\choose k_p-2}-\cdots-{n-(j-2)\choose k_p-2}-{n-j\choose k_p-3} \right]\\
&=\sum_{p\ne i}d_j{n-j\choose k_p-2}.
\end{align*}
Thus, we get
$$\beta(\{2, 3, \dots, j\}, \{2, 3, \dots, j-1\})=\beta(\{2, 3, \dots, j-1\}, \{2, 3, \dots, j-2, j\}).$$
This completes the proof Claim \ref{clm29.1}.
\end{proof}
By (\ref{eq32}) and Claim \ref{clm29.1}, we have
$$\alpha(\{2, 3, \dots, j-1\}, \{2, 3, \dots, j-2, j\})\geq\beta(\{2, 3, \dots, j-1\}, \{2, 3, \dots, j-2, j\}).$$
Note that
$$\{2, 3, \dots, j-2,  j-1\} \overset{1}{\prec} \{2, 3, \dots, j-2, j\}\overset{1}{\prec} \dots  \overset{1}{\prec} \{2, 3, \dots, j-2, n-k_i+j-2\}$$
in $\mathbb{R}_i(k-j+2)$.
By Lemma \ref{clm28}, we have
\begin{align*}
&\alpha(\{2, 3, \dots, j-1\}, \{2, 3, \dots, j-2, n-k_i+j-2\})\\
&>\beta(\{2, 3, \dots, j-1\}, \{2, 3, \dots, j-2, n-k_i+j-2\}),
\end{align*}
that is,
$$\alpha(\{2, 3, \dots, j-1\}, \{2, 3, \dots, j-2\})>\beta(\{2, 3, \dots, j-1\}, \{2, 3, \dots, j-2),$$
or equivalently,
$$f_i(\{2, 3, \dots, j-1\})<f_i(\{2, 3, \dots, j-2\}),$$
as desired.
\end{proof}


\section{Acknowledgements}
 This research is supported by  National natural science foundation of China (Grant No. 11931002 and 12371327).

\end{document}